\documentclass[a4paper,11pt,reqno]{amsart}

\pdfoutput=1

\usepackage{amsmath}
\usepackage{amsthm}
\usepackage{amssymb}
\usepackage{latexsym}
\usepackage{graphicx,color}
\usepackage{booktabs}
\usepackage{hyperref}

\numberwithin{equation}{section}

\newtheorem{thm}{Theorem}[section]
\newtheorem{prop}[thm]{Proposition}
\newtheorem{lem}[thm]{Lemma}

\theoremstyle{definition}
\newtheorem{defn}[thm]{Definition}

\theoremstyle{remark}
\newtheorem{rem}[thm]{Remark}

\allowdisplaybreaks[3]

\newcommand{\al}{\alpha}
\newcommand{\be}{\beta}
\newcommand{\de}{\delta}
\newcommand{\om}{\omega}

\newcommand{\p}{\partial}
\newcommand{\hatp}{\hat{\partial}_x}
\newcommand{\I}{\infty}

\newcommand{\lec}{\lesssim}
\newcommand{\gec}{\gtrsim}

\newcommand{\Sc}[1]{\mathcal{#1}}
\newcommand{\F}{\mathcal{F}}
\renewcommand{\H}{\mathcal{H}}

\newcommand{\Bo}[1]{\mathbb{#1}}
\newcommand{\R}{\mathbb{R}}
\newcommand{\T}{\mathbb{T}}

\newcommand{\hhat}{\widehat}
\newcommand{\bbar}{\overline}
\newcommand{\ti}{\widetilde}

\newcommand{\LR}[1]{{\langle #1 \rangle }}
\newcommand{\chf}[1]{\mathbf{1}_{#1}}

\newcommand{\norm}[2]{\big\| #1 \big\| _{#2}}
\newcommand{\tnorm}[2]{\| #1 \| _{#2}}

\newcommand{\Shugo}[2]{\big\{ \, #1 \, \big| \, #2 \, \big\}}

\newcommand{\eq}[2]{\begin{equation} \label{#1} \begin{split} #2 \end{split} \end{equation}}
\newcommand{\eqq}[1]{\begin{align*} #1 \end{align*}}
\newcommand{\eqs}[1]{\begin{gather*} #1 \end{gather*}}
\newcommand{\mat}[1]{\begin{smallmatrix} #1 \end{smallmatrix}}

\newcommand{\hx}{\hspace{10pt}}
\newcommand{\hxx}{\hspace{30pt}}

\title[Unconditional uniqueness for periodic BO]{Unconditional uniqueness for the periodic Benjamin-Ono equation by normal form approach}
\author[N. Kishimoto]{Nobu Kishimoto}
\address{Research Institute for Mathematical Sciences, Kyoto University, Kyoto 606-8502, Japan}
\email{nobu@kurims.kyoto-u.ac.jp}

\begin{document}

\begin{abstract}
We show unconditional uniqueness of solutions to the Cauchy problem associated with the Benjamin-Ono equation under the periodic boundary condition with initial data given in $H^s$ for $s>1/6$.
This improves the previous unconditional uniqueness result in $H^{1/2}$ by Molinet and Pilod (2012).
Our proof is based on a gauge transform and integration by parts in the time variable.
\end{abstract}

\maketitle


\section{Introduction}

In the present article, we study the Cauchy problem associated with the Benjamin-Ono equation with the periodic boundary condition:
\eq{BO}{\p _tu=-\H \p _x^2u+\p _x(u^2),\qquad (t,x)\in (0,T)\times \T ,}
with initial condition
\eq{BO-id}{u|_{t=0}=u_0,}
where $\T :=\R /2\pi \Bo{Z}$ is the one-dimensional torus, $\H$ denotes the periodic Hilbert transform defined by the Fourier multiplier with symbol $-i\,\mathrm{sgn}(n)\chf{n\neq 0}$, and $u_0$ is a given initial datum.
We only consider the real-valued solutions in this article.
The Benjamin-Ono equation \eqref{BO} is a model for one-dimensional long waves in deep stratified fluids  and is known to be completely integrable.

In this article, we show a uniqueness property of rough solutions to \eqref{BO}.
We begin with some definitions.
For $u\in L^2_{\mathrm{loc}}((0,T)\times \mathbb{T})$, we say $u$ is a solution of the equation \eqref{BO} if $u$ satisfies it in the sense of distributions on $(0,T)\times \mathbb{T}$.
If in addition a solution $u$ lies in $L^\infty ((0,T);H^s)$, where $H^s$ denotes the $L^2$-based Sobolev space, then the limit $u|_{t=0}=\lim _{t\to 0}u(t)$ exists in a weaker topology (see Remark~\ref{rem:weak} below) and the initial condition \eqref{BO-id} can be interpreted in this sense.
We also note that any solution $u$ in $C([0,T];H^s)$, if $s\geq 0$, automatically belongs to $C^1([0,T];H^{s-2})$ and satisfies the equation \eqref{BO} in $H^{s-2}$.
Well-posedness of the Cauchy problem \eqref{BO}--\eqref{BO-id} in $H^s$ means existence of a solution $u$ in $C([0,T];H^s)$, uniqueness of solutions, and continuity in $H^s$ of the solution map $u_0\mapsto u$.
We say the Cauchy problem is locally well-posed if $T$ is chosen as a finite constant (depending on $u_0$), and globally well-posed if $T$ can be arbitrarily large (independent of $u_0$).
Depending on the circumstances, uniqueness may be ensured only under some auxiliary conditions.
Then, uniqueness in the whole of $C([0,T];H^s)$ is called \emph{unconditional uniqueness} (in the literature, it sometimes refers to uniqueness in $L^\infty((0,T);H^s)$).

Let us first recall some of the known results on well-posedness of the Cauchy problem \eqref{BO}--\eqref{BO-id} in $H^s$.
The first result employing the dispersive effect of the linear part in the non-periodic setting (i.e., $x\in \R$) was due to Ponce~\cite{P91}, which was refined by Koch and Tzvetkov~\cite{KT03} and Kenig and Koenig~\cite{KK03}.
Tao~\cite{T04} introduced a technique of a gauge transform in the Benjamin-Ono context, which can transform the original equation containing unfavorable nonlinear interactions of low-high type into an equation with milder nonlinearity.
The idea of the gauge transform was adapted to the method of the Fourier restriction norm and pushed down the regularity threshold for well-posedness.
This was done by Burq and Planchon~\cite{BP08} and by Ionescu and Kenig~\cite{IK07} in the non-periodic case, and by Molinet~\cite{M07,M08} in the periodic case, which established global well-posedness in $L^2$ for both $\R$ and $\T$.
Moreover, Molinet and Pilod~\cite{MP12} gave a simplified proof of these results in $H^s$ for $s\ge 0$, which even complemented them by showing unconditional uniqueness (uniqueness in $L^\I ((0,T);H^s)$) for $s>1/4$ in the non-periodic case and for $s\ge 1/2$ in the periodic case.
Finally, G\'erard et al.~\cite{GKT20p} recently obtained a sharp global well-posedness result for the periodic problem in the whole range of subcritical regularities $s>-1/2$%
\footnote{For $s<0$, the quadratic nonlinearity is not well-defined in the distributional framework, and one needs to extend the notion of solutions and well-posedness. See \cite{GKT20p} for details.} 
by exploiting the complete integrability of the equation.

In this article, we study unconditional uniqueness in the periodic case by a different approach: successive applications of integration by parts with respect to $t$.
This simple technique, while it had been used before to derive nonlinear smoothing effect under the presence of dispersion (see, \mbox{e.g.}, \cite{BMN97}, \cite{TT04}), came to be recognized as a useful tool to establish unconditional uniqueness for nonlinear dispersive equations after a work of Babin \mbox{et al.}~\cite{BIT11} on the periodic Korteweg-de~Vries (KdV) equation and a subsequent work of Kwon and Oh~\cite{KO12} on the modified KdV equation.
It can also be regarded as a variant of the (Poincar\'e-Dulac) normal form reduction in the theory of ordinary differential equations, as pointed out by Guo \mbox{et al.}~\cite{GKO13} who applied the method to the one-dimensional cubic Schr\"odinger equation.
Recently, this method has also been adapted to the non-periodic setting and to problems of different type; see, e.g., \cite{KOY20}, \cite{K-all} and references therein.

Our main result reads as follows:
\begin{thm}\label{thm}
The solution to the Cauchy problem \eqref{BO}--\eqref{BO-id} is unique in the class $C([0,T];H^s(\T ))$ for $s>1/6$.
\end{thm}

\begin{rem}\label{rem:weak}
The uniqueness assertion also holds in the class $L^\I((0,T);H^s(\T ))$, $s>1/6$.
To see this, let $u$ be a solution in $L^\I((0,T);H^s(\T ))$, then $u\in W^{1,\infty}((0,T);H^{s-2})\subset C((0,T);H^{s-2})$ by the equation.
In particular, $u(t)$ has limits in $H^{s-2}$ as $t\to 0,T$ and hence $u\in C([0,T];H^{s-2})$.
By interpolation, for any $s'\in (s-2,s)$ it holds that $\norm{u(t)-u(t')}{H^{s'}}\lec \norm{u(t)-u(t')}{H^{s-2}}^{(s-s')/2}$ for almost every $t,t'\in (0,T)$.
Therefore, $u\in C([0,T];H^{s'})$ for any $s'\in (\max \{ s-2,1/6\},s)$, and Theorem~\ref{thm} can be applied to show uniqueness.
\end{rem}

\begin{rem}
It turns out that twice applications of integration by parts in $t$ (normal form reduction) suffice to prove our result, which is similar to the cases of KdV \cite{BIT11} and modified KdV \cite{KO12} on the torus.
It is, however, not likely that further application of normal form reduction can weaken the regularity restriction.
In fact, we encounter the requirement $s>1/6$ in the control of the \emph{resonant} parts just after the first application (see Lemmas~\ref{lem:N1R}, \ref{lem:N2}, and \ref{lem:N3R} below).
It is not clear whether the threshold $s=1/6$ is optimal or not.
\end{rem}

The plan of this article is as follows.
The proof begins with applying a suitable gauge transform, which is similar to that of \cite{M08} and described in Section~\ref{sec:gauge}, to prevent the derivative in the nonlinearity from falling onto high-frequency components.
In Section~\ref{sec:multi} we give multilinear estimates of various terms arising in the normal form reduction steps.
The main theorem is proved by use of these estimates in Section~\ref{sec:proof}.

At the end, we introduce some notation.
For a Banach space $X$, we abbreviate $C([0,T];X)$ to $C_TX$.
In this article, the Fourier coefficients of a $2\pi$-periodic function $f$ are defined by 
\[ \F f(n):=\frac{1}{2\pi} \int _0^{2\pi}f(x)e^{-inx}\,dx,\qquad n\in \Bo{Z},\]
so that the inverse Fourier transform of a sequence $g=\{ g(n)\}_{n\in \Bo{Z}}$ is given by
\[ \F^{-1}g(x):=\sum _{n\in \Bo{Z}}g(n)e^{inx},\qquad x\in \Bo{T}.\]
We write $P_{\pm}:=\F ^{-1}\chf{\pm n>0}\F$, so that $\H =-iP_++iP_-$, and
\eqs{P_{c}f:=\F ^{-1}\chf{n=0}\F f=\F f(0)=\frac{1}{2\pi}\int _0^{2\pi}f(x)\,dx,\\
P_{\neq c}f:=\F ^{-1}\chf{n\neq 0}\F f=f-\F f(0).
}
For a mean-zero function $f$ on $\T$, define
\eqq{\p _x^{-1}f(x):=\F ^{-1}\Big[ \frac{1}{in}\F f(n)\Big] (x)=\frac{1}{2\pi}\int _0^{2\pi}\int _{\theta}^xf(y)\,dy\,d\theta .}
We also write $\hatp :=P_c+\p _x$.
Note that 
\[ \bbar{P_{\pm}f}=P_{\mp}\bbar{f},\qquad \p _x^{-1}\p _x=\p _x\p _x^{-1}P_{\neq c}=P_{\neq c},\qquad P_{\pm}\hatp =P_{\pm}\p _x,\]
and also that $\hatp$ is a homeomorphism from $H^s(\T )$ to $H^{s-1}(\T )$ for any $s\in \R$ with the inverse mapping given by $\hatp ^{-1}=P_c+\p _x ^{-1}P_{\neq c}$.
Finally, for $s\in \R$ and $1\le p\le \I$ we use the weighted $\ell^p$ space $\ell^p_s(\mathbb{Z})$ defined by
\[ \ell ^p_s(\Bo{Z}):=\Big\{ \om :\Bo{Z}\to \Bo{C} ~\Big| ~\tnorm{\om}{\ell ^p_s}:=\tnorm{\LR{\cdot}^s\om}{\ell ^p}<\infty \Big\} ,\qquad \LR{n}:=(1+n^2)^{1/2}.\]


\bigskip
\section{Gauge transform}\label{sec:gauge}

Let $u\in C_TH^s$ ($s\ge 0$) be a real-valued solution of \eqref{BO} (in the sense of distributions).
First, note that the right-hand side of \eqref{BO} is a well-defined distribution in $C_TH^{-2}$.
Testing the equation against $\phi (t)e^{inx}$ with $\phi \in C_c^\I (0,T)$ and $n\in \Bo{Z}$, we see that $\hhat{u}(t,n)=\F [u(t,\cdot )](n)\in C([0,T])$ satisfies 
\[ \p _t\hhat{u}(t,n)=-in|n|\hhat{u}(t,n)+in\sum _{n'\in \Bo{Z}}\hhat{u}(t,n')\hhat{u}(t,n-n')\qquad \text{in $\Sc{D}'_t(0,T)$}\]
for each $n\in \Bo{Z}$, where the sum is absolutely convergent, and thus the right-hand side is continuous in $t\in [0,T]$.
Hence, $\hhat{u}(\cdot ,n)\in C^1([0,T])$ and the above equation is satisfied in the classical sense.
In particular, the mean value $P_{c}u(t,\cdot )$ is independent of $t$.
In view of the change of the unknown function
\eqq{u(t,x)\quad \mapsto \quad u(t,x-2tP_cu_0)-P_cu_0,}
we may focus without loss of generality on mean-zero solutions, i.e., $P_cu(t,\cdot )=\hhat{u}(t,0)\equiv 0$.

\begin{defn}
We define the gauge transform $v$ of a real-valued mean-zero solution $u\in C_TL^2(\T )$ of \eqref{BO} as follows:
\eq{utov}{V(t):=e^{-i\p _x^{-1}u(t)},\quad v(t):=i\hatp V(t)=iP_c(e^{-i\p _x^{-1}u(t)})+e^{-i\p _x^{-1}u(t)}u(t).}
\end{defn}

Note that $v\in C_TH^s$ if $u\in C_TH^s$ and $s\in [0,1]$, by Lemma~\ref{lem:exp} below.
Moreover, we have
\eq{u_via_v}{&u=P_+u+P_-u=P_+u+\bbar{P_+u},\qquad u=e^{i\p _x^{-1}u}i\p _xe^{-i\p _x^{-1}u}=\bbar{V}P_{\neq c}v, \\
&\qquad\qquad P_+u=P_+\big( \bbar{V}P_{\neq c}v\big) =P_+\big( \bbar{V}P_+v\big) +P_+\big( \bbar{V}P_-v\big) .}

Now, we derive the equation for $v$.
We first consider the gauge transform of regularized solutions $u_N(t):=P_{\le N}u(t)$,%
\footnote{Here, we consider regularization of a solution (which is no longer a solution of \eqref{BO}) rather than the solution with a regularized initial datum, in order to justify the equation for $v$.
This is because general solutions of \eqref{BO} are not necessarily approximated by smooth solutions.
} 
where $P_{\le N}:=\F ^{-1}_n\chf{|n|\le N}\F _x$ for $N>0$.
Set
\eqq{V_N(t):=e^{-i\p _x^{-1}u_N(t)},\quad v_N(t):=i\hatp V_N(t).}
Since the equations
\eqq{\p _tu_N&=-\H \p _x^2u_N+\p _x(u_N^2)+G_N,\\
-i\p _t\p _x^{-1}u_N&=i\H \p _xu_N-iP_{\neq c}(u_N^2)-i\p _x^{-1}G_N}
hold in the classical sense with $G_N:=P_{\le N}\p _x(u^2)-\p _x(u_N^2)$, we have
\eqq{&(\p _t+\H \p _x^2)V_N\\
&=V_N\big( i\H \p _xu_N-iP_{\neq c}(u_N^2)-i\p _x^{-1}G_N\big) +\H \big[ V_N(-i\p _xu_N-u_N^2)\big] \\
&=P_{\neq c}\big[ V_N(i\H )\p _xu_N\big] -(i\H )\big[ V_N\p _xu_N\big] -i(P_{\neq c}-i\H )(V_Nu_N^2) \\
&\hx +iP_c(u_N^2)V_N +P_c\big[ V_N(i\H \p _xu_N-iu_N^2)\big] -iV_N\p _x^{-1}G_N\\
&=-2P_+\big[ V_N\cdot P_-\p _xu_N\big] +2P_-\big[ V_N\cdot P_+\p _xu_N\big] -2iP_-\big[ V_Nu_N^2\big] \\
&\hx +iP_c(u_N^2)V_N +P_c\big[ V_N(i\H \p _xu_N-iu_N^2)\big] -iV_N\p _x^{-1}G_N.
}
Since $\hatp$ commutes with $\p _t+\H  \p _x^2$, we obtain
\eq{eq_vN}{&(\p _t+\H \p _x^2)v_N\\
&=-2iP_+\p _x\big[ V_N\cdot P_-\p _xu_N\big] +2iP_-\p _x\big[ V_N\cdot P_+\p _xu_N\big] +2P_-\p _x\big[ V_Nu_N^2\big] \\
&\hx -P_c(u_N^2)\hatp V_N -P_c\big[ V_N(\H \p _xu_N-u_N^2)\big] +\hatp \big[ V_N\p _x^{-1}G_N\big] .}

The following estimates for the gauge part are easily verified.
\begin{lem}\label{lem:exp}
For $0\le s\le 1$ and real-valued mean-zero $f,g\in H^s$, we have
\eqq{\norm{e^{-i\p _x^{-1}f}}{H^{s+1}}&\lec 1+\tnorm{f}{H^s}^2,\\
\norm{e^{-i\p _x^{-1}f}-e^{-i\p _x^{-1}g}}{H^{s+1}}&\lec \big( 1+\tnorm{f}{H^s}^2+\tnorm{g}{H^s}^2\big) \tnorm{f-g}{H^s}.}
\end{lem}
\begin{proof}
The first inequality with $s=0$ follows from $\norm{e^{-i\p _x^{-1}f}}{L^2}\lec 1$ and $\norm{\p _xe^{-i\p _x^{-1}f}}{L^2}\lec \tnorm{f}{L^2}$.
For $s\in (0,1]$, we deduce 
\[ \norm{\p _xe^{-i\p _x^{-1}f}}{H^s}=\norm{fe^{-i\p _x^{-1}f}}{H^s}\lec \tnorm{f}{H^s}\norm{e^{-i\p _x^{-1}f}}{H^1}\]
from the Sobolev inequality and apply the estimate for $s=0$.
Concerning the second inequality, we notice
\eqq{\norm{e^{-i\p _x^{-1}f}-e^{-i\p _x^{-1}g}}{L^2}&\lec \norm{\p _x^{-1}(f-g)}{L^\I}\lec \tnorm{f-g}{L^2},\\
\norm{\p _x[e^{-i\p _x^{-1}f}-e^{-i\p _x^{-1}g}]}{L^2}&\lec \tnorm{f}{L^2}\norm{\p _x^{-1}(f-g)}{L^\I}+\tnorm{f-g}{L^2}\norm{e^{-i\p _x^{-1}g}}{L^\I}\\
&\lec \big( 1+\tnorm{f}{L^2}\big) \tnorm{f-g}{L^2},
}
which verifies the case $s=0$.
For $s\in (0,1]$, the Sobolev inequality and the estimates for $s=0$ obtained above imply
\eqq{
&\norm{\p _x[e^{-i\p _x^{-1}f}-e^{-i\p _x^{-1}g}]}{H^s}\\
&\lec \tnorm{f}{H^s}\norm{e^{-i\p _x^{-1}f}-e^{-i\p _x^{-1}g}}{H^1}+\tnorm{f-g}{H^s}\norm{e^{-i\p _x^{-1}g}}{H^1}\\
&\lec \tnorm{f}{H^s}\big( 1+\tnorm{f}{L^2})\tnorm{f-g}{L^2}+\big( 1+\tnorm{g}{L^2}\big) \tnorm{f-g}{H^s},
}
which yields the claimed estimate.
\end{proof}

By Lemma~\ref{lem:exp}, we see that $(u_N,v_N,V_N)\to (u,v,V)$ in $C_T(H^s\times H^s\times H^{s+1})$ as $N\to \I$ if $0\le s\le 1$.
Also note that $\p _x^{-1}G_N\to 0$ in $C_TH^{-1}$ as $N\to \I$.
Furthermore, we see 
\eqq{&\norm{P_\pm (fP_\mp \p _xg)}{H^{-1}}=\norm{P_\pm (P_\pm fP_\mp \p _xg)}{H^{-1}}\\
&\le \Big[ \sum _{N_1\sim N_2\gec N\ge 1}+\sum _{N\sim N_1\gg N_2\ge 1}\Big] N^{-1}\norm{P_NP_\pm (P_{N_1}P_\pm f\cdot P_{N_2}P_\mp \p _xg)}{L^2}\\
&\lec \sum _{N_1\sim N_2\ge 1}N_2\norm{P_{N_1}f}{L^2}\norm{P_{N_2}g}{L^2}+\sum _{N_1\gg N_2\ge 1}N_1^{-1}N_2\norm{P_{N_1}f}{L^\I}\norm{P_{N_2}g}{L^2}\\
&\lec \tnorm{f}{H^1}\tnorm{g}{L^2}.}
Hence, if $0\le s\le 1$ the right-hand side of \eqref{eq_vN} converges, in $C_TH^{-2}$ for instance, to
\eqq{&-2iP_+\p _x\big[ V\cdot P_-\p _xu\big] +2iP_-\p _x\big[ V\cdot P_+\p _xu\big] +2P_-\p _x\big[ Vu^2\big] \\
&-P_c(u^2)\hatp V -P_c\big[ V(\H \p _xu-u^2)\big] .}
After substituting
\eqs{P_+\p _xu=P_+\p _x(\bbar{-i\hatp ^{-1}v}P_{\neq c}v),\qquad P_-\p _xu=P_-\p _x(-i\hatp ^{-1}vP_{\neq c}\bar{v}),\\
Vu^2=\bbar{-i\hatp ^{-1}v}(P_{\neq c}v)^2,}
we obtain the equation for $v$ as
\eq{eq_v}{(\p _t+\H \p _x^2)v&=2iP_+\p _x\big[ \hatp ^{-1}v\cdot P_-\p _x(\hatp ^{-1}vP_{\neq c}\bar{v})\big] \\
&\hx +2iP_-\p _x\big[ \hatp ^{-1}v\cdot P_+\p _x(\hatp ^{-1}\bar{v}P_{\neq c}v)\big] \\
&\hx +2iP_-\p _x\big[ \hatp ^{-1}\bar{v}(P_{\neq c}v)^2\big] +R[u],\\
R[u]&:=-P_c(u^2)\hatp V -P_c\big[ V(\H \p _xu-u^2)\big] ,
}
which is satisfied in the sense of distributions.
Note that all the terms but $\p _tv$ are well-defined as distributions in $C_TH^{-2}$.
Taking the Fourier transform in $x$, we obtain the following equations:
\eqq{\p _t\hhat{v}(n)+in|n|\hhat{v}(n)&=\sum _{\mat{n_1,n_2,n_3\in \Bo{Z}\\ n=n_{123}}}m_1(n,n_1,n_2,n_3)\hhat{v}(n_1)\hhat{v}(n_2)\hhat{\bar{v}}(n_3)\\
&\hx +\sum _{\mat{n_1,n_2,n_3\in \Bo{Z}\\ n=n_{123}}}m_2(n,n_1,n_2,n_3)\hhat{v}(n_1)\hhat{\bar{v}}(n_2)\hhat{v}(n_3)\\
&\hx +\sum _{\mat{n_1,n_2,n_3\in \Bo{Z}\\ n=n_{123}}}m_3(n,n_1,n_2,n_3)\hhat{\bar{v}}(n_1)\hhat{v}(n_2)\hhat{v}(n_3)\\
&\hx +\F [R[u]](n),}
where $n_{ijk\ldots}$ stands for $n_i+n_j+n_k+\cdots$, $\hat{n}:=n-i\chf{n=0}$, and
\eqq{m_1(n,n_1,n_2,n_3)&:=2i\frac{nn_{23}}{\hat{n}_1\hat{n}_2}\chf{n>0}\chf{n_{23}<0}\chf{n_3\neq 0},\\
m_2(n,n_1,n_2,n_3)&:=2i\frac{nn_{23}}{\hat{n}_1\hat{n}_2}\chf{n<0}\chf{n_{23}>0}\chf{n_3\neq 0},\\
m_3(n,n_1,n_2,n_3)&:=2i\frac{n}{\hat{n}_1}\chf{n<0}\chf{n_2n_3\neq 0}.}
Notice that
\eqq{
n=n_{123},\quad m_1(n,n_1,n_2,n_3)\neq 0&\qquad \Longrightarrow & n_1~&>~n,-n_{23}~>~0,\\
n=n_{123},\quad m_2(n,n_1,n_2,n_3)\neq 0&\qquad \Longrightarrow & -n_1~&>~-n,n_{23}~>~0.}
In particular, the following bound on multipliers is available:
\eq{bd:multiplier}{n=n_{123}\qquad \Longrightarrow\qquad |m_k(n,n_1,n_2,n_3)|\lec \frac{\LR{n}}{\min\limits _{1\le j\le 3}\LR{n_j}}\qquad (k=1,2,3).}

We now introduce a new unknown function 
\eq{vtoom}{\om (t,n):=e^{itn|n|}\hhat{v}(t,n),}
which obeys the following equation:
\eq{eq_om}{\p _t\om (n)&=\sum _{\mat{n_1,n_2,n_3\in \Bo{Z}\\ n=n_{123}}}e^{it\Phi}\tilde{m}_1(n,n_1,n_2,n_3)\om (n_1)\om (n_2)\om ^*(n_3)\\
&\hx +\sum _{\mat{n_1,n_2,n_3\in \Bo{Z}\\ n=n_{123}}}e^{it\Phi}m_2(n,n_1,n_2,n_3)\om (n_1)\om ^*(n_2)\om (n_3)\\
&\hx +\sum _{\mat{n_1,n_2,n_3\in \Bo{Z}\\ n=n_{123}}}e^{it\Phi}m_3(n,n_1,n_2,n_3)\om ^*(n_1)\om (n_2)\om (n_3)\\
&\hx +\Sc{R}[u,\om ](n)\\
&=:\Sc{N}[\om ](n)+\Sc{R}[u,\om ](n),
}
where
\eqs{\Phi :=n|n|-n_1|n_1|-n_2|n_2|-n_3|n_3|,\\
\tilde{m}_1:=m_1\chf{n_{12}n_{13}\neq 0},\qquad \om ^*(t,n):=\bbar{\om (t,-n)},\\
\begin{split}
\Sc{R}[u,\om ](n):=&\;\sum _{\mat{n_1,n_2,n_3\in \Bo{Z}\\ n=n_{123}}}e^{it\Phi}m_1(n,n_1,n_2,n_3)\chf{n_{12}n_{13}=0}\om (n_1)\om (n_2)\om ^*(n_3)\\
&+e^{itn|n|}\F [R[u]](n).
\end{split}}

Observe that all the sums over $n_1,n_2,n_3$ in the equation \eqref{eq_om} converge absolutely for each $n$ if $\om \in \ell ^2(\Bo{Z})$.
In fact, by \eqref{bd:multiplier} and Young's inequality, for $s\ge 0$ we have
\eqq{\LR{n}^{s-1}|\Sc{N}[\om ](n)|&\lec \sum _{n=n_{123}}\frac{\max\limits _{1\le j\le 3}\LR{n_j}^s}{\min\limits_{1\le j\le 3}\LR{n_j}}|\om (n_1)||\om (n_2)||\om ^*(n_3)|\\
&\lec \tnorm{\om}{\ell ^2_s}\tnorm{\om}{\ell ^1_{-1}}\tnorm{\om }{\ell ^2}\lec \tnorm{\om}{\ell ^2_s}^3.}
In particular, any solution $u\in C_TH^s(\T )$ with $0\le s\le 1$ defines $\om \in C_T\ell ^2_s(\Bo{Z})$ which is $C^1([0,T])$ in $t$ for each $n\in \Bo{Z}$ and satisfies \eqref{eq_om} in the classical sense, with a bound
\eq{est:dtom}{\tnorm{\p _t\om}{C_T\ell ^2_{s-2}}\lec \tnorm{\om}{C_T\ell ^2_s}^3+\tnorm{R[u]}{C_TH^{s-2}}.}
It is easily verified by Lemma~\ref{lem:exp} that $R[u]\in C_TH^s$ as soon as $u\in C_TH^s$ if $0\le s\le 1$.

Note that restricting to positive frequencies reduces the equation \eqref{eq_om} to a simpler one: 
\eq{eq_om+}{\p _t\om (n)&=\sum _{\mat{n_1,n_2,n_3\in \Bo{Z}\\ n=n_{123}}}e^{it\Phi}\tilde{m}_1(n,n_1,n_2,n_3)\om (n_1)\om (n_2)\om ^*(n_3)\\
&\hx +\Sc{R}[u,\om ](n),\qquad\qquad n>0.}
Since we consider real-valued solutions of \eqref{BO}, we only have to bound $\om$ in positive frequencies; see Section~\ref{sec:proof} for details.
However, the full equation \eqref{eq_om} is also needed when we apply normal form reduction.


\bigskip
\section{Normal form reduction and multilinear estimates}\label{sec:multi}

\subsection{First normal form reduction}
Let us start with the equation \eqref{eq_om+}, restricting to positive frequencies.
Let $M>0$ be a large constant, which will be chosen at the end of the proof of uniqueness depending on the solutions we consider.
Decompose the equation into ``resonant'' and ``non-resonant'' parts as
\eqq{\p _t\om (n)&=\Big[ \sum _{\mat{n_1,n_2,n_3\in \Bo{Z}\\ n=n_{123}\\ |\Phi |\le M}}+\sum _{\mat{n_1,n_2,n_3\in \Bo{Z}\\ n=n_{123}\\ |\Phi |>M}}\Big] e^{it\Phi}\tilde{m}_1(n,n_1,n_2,n_3)\om (n_1)\om (n_2)\om ^*(n_3)+\Sc{R}[u,\om ](n)\\
&=:\Sc{N}_R[\om ](n)+\Sc{N}_{N\!R}[\om ](n)+\Sc{R}[u,\om ](n),\qquad n>0.}

The estimate for $\Sc{R}$ is easy.
\begin{lem}\label{lem:R}
Let $0\le s\le 1$.
Then, we have
\eqq{\norm{\Sc{R}[u,\om ]}{C_T\ell ^2_s}&\lec 1+\tnorm{u}{C_TH^s}^4+\tnorm{\om}{C_T\ell ^2_s}^3,\\
\norm{\Sc{R}[u,\om ]-\Sc{R}[\ti{u},\ti{\om}]}{C_T\ell ^2_s}&\lec \Big( 1+\tnorm{u}{C_TH^s}^4+\tnorm{\ti{u}}{C_TH^s}^4\Big) \tnorm{u-\ti{u}}{C_TH^s}\\
&\quad +\Big( \tnorm{\om}{C_T\ell ^2_s}^2+\tnorm{\ti{\om}}{C_T\ell ^2_s}^2\Big) \tnorm{\om -\ti{\om}}{C_T\ell ^2_s}}
for any real-valued mean-zero functions $u,\ti{u}\in C_TH^s$ and any $\om ,\ti{\om}\in C_T\ell ^2_s$.
\end{lem}
\begin{proof}
Observing $|m_1(n_{123},n_1,n_2,n_3)|\chf{n_{12}n_{13}=0}\lec 1$, the $C_T\ell ^2_s$-norm of the first term of $\Sc{R}[u,\om ]$ is bounded by $\tnorm{\om}{C_T\ell ^2}^2\tnorm{\om}{C_T\ell ^2_s}$ for any $s\ge 0$.
Estimates on $R[u]$ in $C_TH^s$ can be easily obtained from Lemma~\ref{lem:exp} when $0\le s\le 1$.
\end{proof}

The resonant part $\Sc{N}_R$ is also controlled in $H^s$ if $s\ge 0$.
\begin{lem}\label{lem:N_R}
Let $s\ge 0$.
Then, we have
\eqq{\norm{\Sc{N}_R[\om ]}{C_T\ell ^2_s}&\lec M\tnorm{\om}{C_T\ell ^2_s}^3,\\
\norm{\Sc{N}_R[\om ]-\Sc{N}_R[\ti{\om}]}{C_T\ell ^2_s}&\lec M\Big( \tnorm{\om}{C_T\ell ^2_s}^2+\tnorm{\ti{\om}}{C_T\ell ^2_s}^2\Big) \tnorm{\om -\ti{\om}}{C_T\ell ^2_s}}
for any $\om ,\ti{\om} \in C_T\ell ^2_s$.
\end{lem}
\begin{proof}
First of all, we claim that
\eq{claim:Phi}{|\Phi |\gec \begin{cases}
|n_{12}||n_{23}| &\text{if $|n_2|\ge |n_3|$},\\
|n_{13}||n_{23}| &\text{if $|n_2|<|n_3|$}
\end{cases}}
whenever $n=n_{123}$ and $\tilde{m}_1(n,n_1,n_2,n_3)\neq 0$.
In fact, it holds that $n_1>n>0$ and $-n_1<n_{23}<0$, and we have
\eq{id:Phi}{\Phi &=n^2-n_1^2-n_2|n_2|-n_3|n_3|\\
&=\begin{cases}
n^2-n_1^2-n_2^2+n_3^2=2n_{13}n_{23} &\text{if $n_2\ge 0$ and $n_3<0$},\\
n^2-n_1^2+n_2^2-n_3^2=2n_{12}n_{23} &\text{if $n_2<0$ and $n_3\ge 0$},\\
n^2-n_1^2+n_2^2+n_3^2=2(nn_{23}-n_2n_3) &\text{if $n_2,n_3<0$}.\end{cases}}
In the first two cases the claim follows from $n_{23}<0$.
In the third case, we see that both $nn_{23}$ and $-n_2n_3$ are negative and $|\Phi|\gec |n||n_{23}|\vee |n_2||n_3|$.
Let us assume $|n_2|\ge |n_3|$ by symmetry.
The claim follows if $|n|\gec |n_{12}|$, while if $|n|\ll |n_{12}|$ we have $|n_{12}|\sim |n_3|$ and $|n_{23}|\sim |n_2|$, which implies \eqref{claim:Phi}.

We observe that \eqref{claim:Phi} yields the bound
\eqq{\bigg| \frac{\LR{n}^s\tilde{m}_1(n,n_1,n_2,n_3)}{\Phi}\bigg| \lec \frac{\chf{n_{12}n_{13}\neq 0}\LR{n_1}^s}{\LR{n_2}(|n_{12}|\wedge |n_{13}|)},}
and the desired estimates are reduced to showing
\eq{est1}{\norm{\sum _{n=n_{123}}\frac{\LR{n_1}^s\om _1(n_1)\om _2(n_2)\om _3(n_3)}{\LR{n_2}[\LR{n_{12}}\wedge \LR{n_{13}}]}}{\ell ^2}\lec \tnorm{\om _1}{\ell ^2_s}\tnorm{\om _2}{\ell ^2_s}\tnorm{\om _3}{\ell ^2_s}}
for $\om _1,\om _2,\om _3\in \ell ^2_s$.
If $|n_{12}|\ge |n_{13}|$, then by Young's inequality we have
\eqq{\text{LHS of \eqref{est1}}&\lec \norm{\Big( \frac{\om _2}{\LR{\cdot}}\Big) * \Big( \frac{\big( \LR{\cdot}^s\om _1\big) *\om _3}{\LR{\cdot}}\Big)}{\ell ^2}\le \tnorm{\om _2}{\ell ^1_{-1}}\norm{\big( \LR{\cdot}^s\om _1\big) *\om _3}{\ell ^2_{-1}}\\
&\le \tnorm{\om _2}{\ell ^2}\norm{\big( \LR{\cdot}^s\om _1\big) *\om _3}{\ell ^\I}\le \tnorm{\om _1}{\ell ^2_s}\tnorm{\om _2}{\ell ^2}\tnorm{\om _3}{\ell ^2}.}
If $|n_{12}|<|n_{13}|$, then we have
\eqq{\text{LHS of \eqref{est1}}&\lec \norm{\Big( \frac{\big( \LR{\cdot}^s\om _1\big) *\big( \LR{\cdot}^{-1}\om _2\big)}{\LR{\cdot}}\Big) *\om _3}{\ell ^2}\\
&\le \norm{\big( \LR{\cdot}^s\om _1\big) *\big( \LR{\cdot}^{-1}\om _2\big)}{\ell ^1_{-1}}\tnorm{\om _3}{\ell ^2}\\
&\le \tnorm{\LR{\cdot}^s\om _1}{\ell ^2}\tnorm{\LR{\cdot}^{-1}\om _2}{\ell ^1}\tnorm{\om _3}{\ell ^2}\le \tnorm{\om _1}{\ell ^2_s}\tnorm{\om _2}{\ell ^2}\tnorm{\om _3}{\ell ^2}.\qedhere}
\end{proof}

The remaining term $\Sc{N}_{N\!R}[u]$, which can be controlled in $\ell ^2_s$ if $s>3/2$, does not seem to admit an $\ell ^2_s$ estimate for lower regularities.
Therefore, we apply a differentiation by parts, noting that $\Phi$ remains non-zero due to \eqref{claim:Phi}:
\eqq{&\Sc{N}_{N\!R}[\om ](n)\\
&=-i\sum _{\mat{n=n_{123}\\ |\Phi |>M}}\tilde{m}_1(n,n_1,n_2,n_3)\bigg[ \p _t\Big( \frac{e^{it\Phi}}{\Phi}\om (n_1)\om (n_2)\om ^*(n_3)\Big) -\frac{e^{it\Phi}}{\Phi}\p _t\Big( \om (n_1)\om (n_2)\om ^*(n_3)\Big) \bigg] \\
&=:\p _t\Sc{N}_0[\om ](n)+\Sc{N}_1[\om ](n)+\Sc{N}_2[\om ](n)+\Sc{N}_3[\om ](n)+\Sc{R}_1[u,\om ](n),}
where we change the order of summation and time differentiation and apply the product rule, and then substitute the equation \eqref{eq_om} to obtain the expression of each term as
\eqq{\Sc{N}_0[\om ](n)&:=-i\sum _{\mat{n=n_{123}\\ |\Phi |>M}}\frac{e^{it\Phi}\tilde{m}_1(n,n_1,n_2,n_3)}{\Phi}\om (n_1)\om (n_2)\om ^*(n_3),\\
\Sc{N}_1[\om ](n)&:=i\sum _{\mat{n=n_{123}\\ |\Phi |>M}}\frac{e^{it\Phi}\tilde{m}_1(n,n_1,n_2,n_3)}{\Phi}\Sc{N}[\om ](n_1)\om (n_2)\om ^*(n_3),\\
\Sc{N}_2[\om ](n)&:=i\sum _{\mat{n=n_{123}\\ |\Phi |>M}}\frac{e^{it\Phi}\tilde{m}_1(n,n_1,n_2,n_3)}{\Phi}\om (n_1)\Sc{N}[\om ](n_2)\om ^*(n_3),\\
\Sc{N}_3[\om ](n)&:=i\sum _{\mat{n=n_{123}\\ |\Phi |>M}}\frac{e^{it\Phi}\tilde{m}_1(n,n_1,n_2,n_3)}{\Phi}\om (n_1)\om (n_2)\bbar{\Sc{N}[\om ](-n_3)},}
and
\eqq{\Sc{R}_1[u,\om ](n):=&i\sum _{\mat{n=n_{123}\\ |\Phi |>M}}\frac{e^{it\Phi}\tilde{m}_1(n,n_1,n_2,n_3)}{\Phi}\Big[ \Sc{R}[u,\om ](n_1)\om (n_2)\om ^*(n_3)\\
&\qquad +\om (n_1)\Sc{R}[u,\om ](n_2)\om ^*(n_3)+\om (n_1)\om (n_2)\bbar{\Sc{R}[u,\om ](-n_3)}\Big] .
}

We can justify the above formal computations as follows.
The change of summation and time differentiation in $\Sc{N}_0$ is justified in the framework of distributions (see, e.g., \cite[Lemma~5.1]{GKO13}), since the sum is absolutely convergent in view of \eqref{est:dtom}.
Application of the product rule also makes sense because $\om (\cdot ,n)\in C^1$ for each $n$ if $s\ge 0$. 
Absolute convergence of all the sums in $\Sc{N}_1$, $\Sc{N}_2$, and $\Sc{N}_3$, which justifies substitution of the equation \eqref{eq_om}, will be proved at the end of this subsection.

The estimate for $\Sc{R}_1$ is immediately obtained as a corollary of Lemma~\ref{lem:N_R}.
\begin{lem}\label{lem:R1}
Let $0\le s\le 1$.
Then, we have
\eqq{&\norm{\Sc{R}_1[u,\om ]}{C_T\ell ^2_s}\lec \Big( 1+\tnorm{u}{C_TH^s}^4+\tnorm{\om}{C_T\ell ^2_s}^3\Big) \tnorm{\om}{C_T\ell ^2_s}^2,\\
&\norm{\Sc{R}_1[u,\om ]-\Sc{R}_1[\ti{u},\ti{\om}]}{C_T\ell ^2_s}\\
&\hx \lec \Big( 1+\tnorm{u}{C_TH^s}^4+\tnorm{\ti{u}}{C_TH^s}^4\Big) \Big( \tnorm{\om}{C_T\ell ^2_s}^2+\tnorm{\ti{\om}}{C_T\ell ^2_s}^2\Big) \tnorm{u-\ti{u}}{C_TH^s}\\
&\hx \hx +\Big( 1+\tnorm{u}{C_TH^s}^4+\tnorm{\ti{u}}{C_TH^s}^4+\tnorm{\om}{C_T\ell ^2_s}^3+\tnorm{\ti{\om}}{C_T\ell ^2_s}^3\Big) \Big( \tnorm{\om}{C_T\ell ^2_s}+\tnorm{\ti{\om}}{C_T\ell ^2_s}\Big) \tnorm{\om -\ti{\om}}{C_T\ell ^2_s}}
for any real-valued mean-zero functions $u,\ti{u}\in C_TH^s$ and any $\om ,\ti{\om}\in C_T\ell ^2_s$.
\end{lem}
\begin{proof}
This is just a combination of Lemma~\ref{lem:R} and the proof of Lemma~\ref{lem:N_R}.
\end{proof}

Due to the time derivative on $\Sc{N}_0$, the estimate for it requires an extra small factor.
We assume $s>0$ so that a negative power of $M$ can be included in the estimate.
\begin{lem}\label{lem:N0}
Let $s>0$.
Then, there exists $\de >0$ such that we have
\eqq{\norm{\Sc{N}_0[\om ]}{C_T\ell ^2_s}&\lec M^{-\de}\tnorm{\om}{C_T\ell ^2_s}^3,\\
\norm{\Sc{N}_0[\om ]-\Sc{N}_0[\ti{\om}]}{C_T\ell ^2_s}&\lec M^{-\de}\Big( \tnorm{\om}{C_T\ell ^2_s}^2+\tnorm{\ti{\om}}{C_T\ell ^2_s}^2\Big) \tnorm{\om -\ti{\om}}{C_T\ell ^2_s}}
for any $\om ,\ti{\om}\in C_T\ell ^2_s$.
\end{lem}
\begin{proof}
The claim is verified by modifying the proof of Lemma~\ref{lem:N_R}.
We divide the sum in \eqref{est1} into two parts as follows.

If $\LR{n_2}[\LR{n_{12}}\wedge \LR{n_{13}}]\gec \max _{1\le j\le 3}\LR{n_j}$, then it suffices to observe that \eqref{est1} still holds if the denominator in the left-hand side is replaced with $\LR{n_2}^{(1/2)+}[\LR{n_{12}}\wedge \LR{n_{13}}]^{(1/2)+}$, and that $\LR{n_2}[\LR{n_{12}}\wedge \LR{n_{13}}]\gec |\Phi |^{1/2}>M^{1/2}$.

Suppose that $\LR{n_2}[\LR{n_{12}}\wedge \LR{n_{13}}]\ll \max _{1\le j\le 3}\LR{n_j}$.
This in particular implies $\LR{n_2}\ll \LR{n_1}\vee \LR{n_3}$.
If $\LR{n_1}\gg \LR{n_3}$, then $\LR{n_1}\sim \LR{n_{12}}\sim \LR{n_{13}}\ll \max _{1\le j\le 3}\LR{n_j}=\LR{n_1}$ --- a contradiction.
Hence, it must hold $\LR{n_3}\sim \max _{1\le j\le 3}\LR{n_j}\gec |\Phi |^{1/2}>M^{1/2}$ and then $\tnorm{\om _3}{\ell ^2}\lec M^{-s/2}\tnorm{\om _3}{\ell ^2_s}$ if $s>0$.
\end{proof}

The remaining terms $\Sc{N}_1$, $\Sc{N}_2$, $\Sc{N}_3$ will be treated in the subsequent subsections.
Finally, let us see that the sums in these terms are absolutely convergent.
This is the first step where the regularity threshold $s=1/6$ appears.
\begin{lem}\label{lem:weak-Nj}
Let $1/6\le s<1/2$ and $\al <3s-1$.
Then, we have
\eqq{\norm{\Sc{N}_j[\om ]}{C_T\ell ^2_{\al}}\lec \tnorm{\om}{C_T\ell ^2_s}^5}
for $j=1,2,3$.
\end{lem}

\begin{proof}
We first upgrade \eqref{est:dtom} as follows: for $0<s<1/2$, it holds
\eq{est:dtom2}{\norm{\Sc{N}[\om ]}{C_T\ell ^2_{\be}}\lec \tnorm{\om}{C_T\ell ^2_s}^3,\qquad \be :=2s-\frac{3}{2}\in (-\frac{3}{2},-\frac{1}{2}).}
By \eqref{bd:multiplier}, the estimate \eqref{est:dtom2} is reduced to showing
\[ \tnorm{fgh}{H^{\be +1}}\lec \tnorm{f}{H^s}\tnorm{g}{H^s}\tnorm{h}{H^{s+1}}.\]
This is shown by applying the Sobolev multiplication law (see \mbox{e.g.}  \cite[Lemma~3.4]{GLM14}) twice:
\[ \tnorm{fgh}{H^{2s-\frac{1}{2}}}\lec \tnorm{f}{H^s}\tnorm{gh}{H^s}\lec \tnorm{f}{H^s}\tnorm{g}{H^s}\tnorm{h}{H^{s+1}},\]
where in the first inequality we have used the assumption $0<s<1/2$.

Let $0<s<1/2$.
In view of \eqref{est:dtom2} it suffices to show
\eqq{\norm{\sum _{n=n_{123}}\frac{\tilde{m}_1(n,n_1,n_2,n_3)\LR{n}^\al n_{\max}^{s-\be}}{\Phi \LR{n_1}^s\LR{n_2}^s\LR{n_3}^s}\om _1(n_1)\om _2(n_2)\om _3(n_3)}{\ell ^2}\lec \tnorm{\om _1}{\ell ^2}\tnorm{\om _2}{\ell ^2}\tnorm{\om _3}{\ell ^2},}
where $n_{\max}:=\max _{1\le j\le 3}\LR{n_j}$.
From the definition of $\tilde{m}_1$ and \eqref{claim:Phi} we have
\eq{est:mPhi}{\Big| \frac{\tilde{m}_1(n,n_1,n_2,n_3)\LR{n}^\al n_{\max}^{s-\be}}{\Phi \LR{n_1}^s\LR{n_2}^s\LR{n_3}^s}\Big| \lec \frac{\LR{n}^{1+\al}n_{\max}^{s-\be}}{\LR{n_1}^{1+s}\LR{n_2}^{1+s}\LR{n_3}^s[\LR{n_{12}}\wedge \LR{n_{13}}]},}
and the following four cases are possible to occur.

(i) $\LR{n_1}\gg \LR{n_2},\LR{n_3}$.
In this case we have $\LR{n}\sim \LR{n_1}\sim [\LR{n_{12}}\wedge \LR{n_{13}}]\sim n_{\max}$ and 
\eqq{\eqref{est:mPhi}\sim \frac{\LR{n_1}^{\al -\be -1}}{\LR{n_2}^{1+s}\LR{n_3}^s}\lec \begin{cases}
\dfrac{\LR{n_1}^{\al -\be -1}}{\LR{n_2}^{s+1/2}\LR{n_3}^{s+1/2}} &\text{if $\LR{n_2}\gec \LR{n_3}$},\\[15pt]
\dfrac{\LR{n_1}^{\al -\be -s-(1/2)+}}{\LR{n_2}^{1+s}\LR{n_3}^{(1/2)+}} &\text{if $\LR{n_2}\ll \LR{n_3}$.} \end{cases}
}
Hence, by Young's inequality, this case is handled whenever 
\eqq{\al <\left\{ \begin{array}{lll}
\be +1 &=2s-1/2 ~(>3s-1)&\hx \text{if $\LR{n_2}\gec \LR{n_3}$},\\
\be +s+1/2 &=3s-1 &\hx \text{if $\LR{n_2}\ll \LR{n_3}$}.\end{array}\right.}

(ii) $\LR{n_1}\sim \LR{n_2}\gec \LR{n_3}$.
Since $\LR{n}\le \LR{n_1}$ whenever $\tilde{m}_1\neq 0$, we have $\LR{n}=\LR{n_3+n_{12}}\le \LR{n_3}+\LR{n_{12}}$ and $\LR{n}\le \LR{n_3-n_{13}}\le \LR{n_3}+\LR{n_{13}}$, which imply $\LR{n}\le \LR{n_3}+[\LR{n_{12}}\wedge \LR{n_{13}}]$.
Hence, we have
\eqq{\eqref{est:mPhi}&\sim \frac{\LR{n}^{1+\al}n_1^{-\be -s-2}}{\LR{n_3}^s[\LR{n_{12}}\wedge \LR{n_{13}}]}\lec \frac{\LR{n}^{1+\al}n_1^{-\be -s-2+}}{\LR{n_3}^{s}[\LR{n_{12}}\wedge \LR{n_{13}}]^{1+}}\\
&\lec \frac{\LR{n}^{1+\al -s}n_1^{-\be -s-2+}}{[\LR{n_{12}}\wedge \LR{n_{13}}]^{1+}}+\frac{\LR{n}^{1+\al -s}n_1^{-\be -s-2+}}{\LR{n_3}^{s}[\LR{n_{12}}\wedge \LR{n_{13}}]^{1-s+}}.}
This yields the desired estimate if $-\be -s-2<0$ and $\al -\be -2s-1<0$, which is possible if $\al <4s-1/2~(>3s-1)$.

(iii) $\LR{n_1}\sim \LR{n_3}\gg \LR{n_2}$.
The worst interaction occurs in this case. 
Since $\LR{n}\le \LR{n_2}+\LR{n_{13}}$, we see that 
\eqq{\eqref{est:mPhi}\sim \frac{\LR{n}^{1+\al}n_1^{-\be -s-1}}{\LR{n_2}^{1+s}\LR{n_{13}}}\lec \begin{cases}
\dfrac{\LR{n}^{\al -s+}n_1^{-\be -s-1}}{\LR{n_2}^{0+}\LR{n_{13}}} &\text{if $\LR{n}\lec \LR{n_2}$},\\[15pt]
\dfrac{\LR{n}^{(1/2)+\al +}n_1^{-\be -s-1}}{\LR{n_2}^{1+s}\LR{n_{13}}^{(1/2)+}} &\text{if $\LR{n}\gg \LR{n_2}$}.\end{cases}}
This yields the desired estimate if $-\be -s-1\le 0$ (i.e., $s\ge 1/6$) and 
\eqq{\al <\left\{ \begin{array}{lll}
\be +2s+1 &=4s-1/2 &\hx \text{if $\LR{n}\lec \LR{n_2}$},\\
\be +s+1/2 &=3s-1 &\hx \text{if $\LR{n}\gg \LR{n_2}$}.\end{array}\right.}

(iv) $\LR{n_2}\sim \LR{n_3}\gg \LR{n_1}$.
Recalling $\LR{n}\le \LR{n_1}$, we have
\eqq{\eqref{est:mPhi}\sim \frac{\LR{n}^{1+\al}n_2^{-\be -s-2}}{\LR{n_1}^{1+s}}\lec \frac{\LR{n}^{1+\al -s+}n_2^{-\be -s-2}}{\LR{n}^{(1/2)+}\LR{n_1}^{(1/2)+}}.}
Therefore, the claim follows if $-\be -s-2\le 0$ and $\al -\be -2s-1<0$, which is possible if $\al <4s-1/2$.
\end{proof}
\begin{rem}\label{rem1}
From the above proof, we see that the condition on $\al$ is weakened to $\al <4s-1/2$ if the sum is restricted to frequencies such that $\LR{n_2}\gec \LR{n}$.
In particular, $\Sc{N}_j[\om ]$ is controlled in $\ell ^2_s$ for $1/6<s<1/2$ in this frequency range, since we can choose $\al =s$ if $s>1/6$.
\end{rem}


\subsection{Estimate for $\Sc{N}_1$}

By the positivity of $n_1$, $\Sc{N}_1[\om ]$ can be written as
\eqq{i\sum _{\mat{n=n_{123}\\ |\Phi |>M}}\sum _{n_1=n_{456}}\Big[ \frac{e^{it(\Phi +\Phi _1)}\tilde{m}_1(n,n_1,n_2,n_3)\tilde{m}_1(n_1,n_4,n_5,n_6)}{\Phi}\om (n_2)\om ^*(n_3)\om (n_4)\om (n_5)\om ^*(n_6)\Big] ,}
where $\Phi _1:=n_1|n_1|-n_4|n_4|-n_5|n_5|-n_6|n_6|$.
We decompose $\Sc{N}_1$ as
\eqq{\Sc{N}_1[\om ](n)&=i\sum _{\mat{n=n_{123}\\ |\Phi |>M}}\sum _{n_1=n_{456}}\Big( \chf{\Sc{A}_1^c}+\chf{\Sc{A}_1}\Big) \Big[ \cdots \Big] =:\Sc{N}_{1,R}[\om ](n)+\Sc{N}_{1,N\!R}[\om ](n),}
where $\Sc{A}_1$, a subset of $\Shugo{(n,n_1,\dots ,n_6)\in \Bo{Z}^7}{n=n_{123},\,n_1=n_{456}}$, is defined by
\eqq{\Sc{A}_1:= &\{ \,(n,n_j):\LR{n}\gg \LR{n_2}\gec \LR{n_3}~~\text{and}~~\LR{n_5}\ll \LR{n_1}\wedge \LR{n_6}~~\text{and}~~\LR{n_2}\ll \LR{n_6}\} \\
&\cup \{ \,(n,n_j):\LR{n_2}\ll \LR{n},\LR{n_3}~~\text{and}~~\LR{n_5}\ll \LR{n_1}\wedge \LR{n_6} \} .}

We show that $\Sc{N}_{1,R}$ is controlled in $\ell ^2_s$, $s>1/6$.
\begin{lem}\label{lem:N1R}
Let $1/6<s<1/2$.
Then, we have
\eqq{\norm{\Sc{N}_{1,R}[\om ]}{C_T\ell ^2_s}&\lec \tnorm{\om}{C_T\ell ^2_s}^5,\\
\norm{\Sc{N}_{1,R}[\om ]-\Sc{N}_{1,R}[\ti{\om}]}{C_T\ell ^2_s}&\lec \Big( \tnorm{\om}{C_T\ell ^2_s}^4+\tnorm{\ti{\om}}{C_T\ell ^2_s}^4\Big) \tnorm{\om -\ti{\om}}{C_T\ell ^2_s}}
for any $\om ,\ti{\om}\in C_T\ell ^2_s$.
\end{lem}

\begin{proof}
In view of Remark~\ref{rem1}, it suffices to deal with the case $\LR{n}\gg \LR{n_2}$.%
\footnote{We easily see that Lemma~\ref{lem:weak-Nj} and the estimate \eqref{est:dtom2} admit the corresponding difference estimates.}
We note that this restriction implies $\LR{n}\LR{n_{23}}\lec |\Phi |$ by virtue of \eqref{claim:Phi}.
According to the definition of $\Sc{A}_1$, consider the following two cases.

i) $\LR{n}\gg \LR{n_2}$ and $\LR{n_5}\gec \LR{n_1}\wedge \LR{n_6}$.
In this case, we have $|\tilde{m}_1(n_1,n_4,n_5,n_6)|\lec 1$.
Applying the Sobolev multiplication law twice, we obtain that
\eqq{&\norm{\sum _{n_1=n_{456}}e^{it\Phi _1}\tilde{m}_1(n_1,n_4,n_5,n_6)\om _4(n_4)\om _5(n_5)\om _6(n_6)}{\ell ^2_{3s-1}}\\
&\lec \tnorm{\om _4*\om _5*\om _6}{\ell ^2_{3s-1}}\lec \tnorm{\om _4}{\ell ^2_s}\tnorm{\om _5*\om _6}{\ell ^2_{2s-1/2}}\lec \tnorm{\om _4}{\ell ^2_s}\tnorm{\om _5}{\ell ^2_s}\tnorm{\om _6}{\ell ^2_s}}
for $1/6<s<1/2$.
Repeating the proof of Lemma~\ref{lem:weak-Nj} with the above estimate instead of \eqref{est:dtom2}, we obtain a weaker condition on $\al$ that $\al <4s-1/2$ in Cases (i) and (iii).
This suffices for the claim since $s<4s-1/2$.

ii) $\LR{n}\gg \LR{n_2}\gec \LR{n_3}\vee \LR{n_6}$.
Using \eqref{claim:Phi} and $\LR{n}\sim [\LR{n_{12}}\wedge \LR{n_{13}}]\sim \LR{n_1}<\LR{n_4}$, we see that
\eqq{&\Big| \frac{\LR{n}^s\tilde{m}_1(n,n_1,n_2,n_3)\tilde{m}_1(n_1,n_4,n_5,n_6)}{\Phi}\Big| \\&\lec \frac{\LR{n_4}^s}{\LR{n_2}\LR{n_5}}\lec \frac{\LR{n_4}^s\LR{n_2}^{(1/6)+}\LR{n_3}^{(1/6)+}\LR{n_6}^{(1/6)+}}{\LR{n_5}\LR{n_2}^{(1/2)+}\LR{n_3}^{(1/2)+}\LR{n_6}^{(1/2)+}},}
which yields the claim if $s>1/6$ by use of Young's inequality.
\end{proof}

For $\Sc{N}_{1,N\!R}$, we apply the normal form reduction argument again.
Let us begin with observing that we have $|\Phi +\Phi _1|\gec |\Phi _1|$ for frequencies included in $\Sc{N}_{1,N\!R}$.
In fact, if $\LR{n}\gg \LR{n_2}\gec \LR{n_3}$, $\LR{n_5}\ll \LR{n_1},\LR{n_6}$, and $\LR{n_2}\ll \LR{n_6}$, then we see from \eqref{id:Phi} and \eqref{claim:Phi} that $|\Phi |\lec |n_1||n_2|\ll |n_1-n_5||n_6|\sim |n_{46}||n_{56}|\lec |\Phi _1|$, whereas if $\LR{n_2}\ll \LR{n},\LR{n_3}$ and $\LR{n_5}\ll \LR{n_1},\LR{n_6}$, we have $n_3,n_6<0$ and $n_{13}=n-n_2>0$, $n_{46}=n_1-n_5>0$, which combined with \eqref{id:Phi} imply that $\Phi <0$ and $\Phi _1<0$, and therefore $|\Phi +\Phi _1|>|\Phi _1|$.

Noting $|\Phi +\Phi _1|\gec |\Phi _1|>0$ by \eqref{claim:Phi}, we take a differentiation by parts as
\eqq{&\Sc{N}_{1,N\!R}[\om ](n)\\
&=\p _t\bigg[ \sum _{\mat{n=n_{123}\\ |\Phi |>M}}\sum _{n_1=n_{456}}\frac{e^{it(\Phi +\Phi _1)}\chf{\Sc{A}_1}\tilde{m}_1(n,n_1,n_2,n_3)\tilde{m}_1(n_1,n_4,n_5,n_6)}{\Phi (\Phi +\Phi _1)}\\[-10pt]
&\hspace{180pt} \times \om (n_2)\om ^*(n_3)\om (n_4)\om (n_5)\om ^*(n_6)\bigg] \\
&\hx -\sum _{\mat{n=n_{123}\\ |\Phi |>M}}\sum _{n_1=n_{456}}\frac{e^{it(\Phi +\Phi _1)}\chf{\Sc{A}_1}\tilde{m}_1(n,n_1,n_2,n_3)\tilde{m}_1(n_1,n_4,n_5,n_6)}{\Phi (\Phi +\Phi _1)}\\[-10pt]
&\hspace{170pt} \times \p _t\Big[ \om (n_2)\om ^*(n_3)\om (n_4)\om (n_5)\om ^*(n_6)\Big] \\
&=:\p _t\Sc{N}_{1,0}[\om ](n)+\Sc{N}_{1,1}[\om ](n).}
The above formal computations are again justified by the absolute convergence of the sum in $\Sc{N}_{1,0}$, which follows from Lemma~\ref{lem:weak-Nj}.

Now, we give estimates for $\Sc{N}_{1,1}$.
\begin{lem}\label{lem:N11}
Let $1/6<s<1/2$.
Then, we have
\eqq{&\norm{\Sc{N}_{1,1}[\om ]}{C_T\ell ^2_s}\lec \Big( 1+\tnorm{u}{C_TH^s}^4+\tnorm{\om}{C_T\ell ^2_s}^3\Big) \tnorm{\om}{C_T\ell ^2_s}^4,\\
&\norm{\Sc{N}_{1,1}[\om ]-\Sc{N}_{1,1}[\ti{\om}]}{C_T\ell ^2_s}\\
&\hx \lec \Big( 1+\tnorm{u}{C_TH^s}^4+\tnorm{\ti{u}}{C_TH^s}^4\Big) \Big( \tnorm{\om}{C_T\ell ^2_s}^4+\tnorm{\ti{\om}}{C_T\ell ^2_s}^4\Big) \tnorm{u-\ti{u}}{C_TH^s}\\
&\hx \hx +\Big( 1+\tnorm{u}{C_TH^s}^4+\tnorm{\ti{u}}{C_TH^s}^4+\tnorm{\om}{C_T\ell ^2_s}^3+\tnorm{\ti{\om}}{C_T\ell ^2_s}^3\Big) \Big( \tnorm{\om}{C_T\ell ^2_s}^3+\tnorm{\ti{\om}}{C_T\ell ^2_s}^3\Big) \tnorm{\om -\ti{\om}}{C_T\ell ^2_s}}
for any real-valued mean-zero solutions $u,\ti{u}\in C_TH^s$ of \eqref{BO} and the corresponding solutions $\om ,\ti{\om}\in C_T\ell ^2_s$ of \eqref{eq_om} defined by \eqref{utov}, \eqref{vtoom}.
\end{lem}

\begin{proof}
In view of Lemma~\ref{lem:R} and \eqref{est:dtom2} it suffices to show that
\eqq{\Big\| \sum _{\mat{n=n_{123}\\ |\Phi |>M}}\sum _{n_1=n_{456}}&\Big| \frac{\chf{\Sc{A}_1}\tilde{m}_1(n,n_1,n_2,n_3)\tilde{m}_1(n_1,n_4,n_5,n_6)n_{\max}^{s-(2s-3/2)}}{\Phi (\Phi +\Phi _1)\LR{n_2}^s\LR{n_3}^s\LR{n_4}^s\LR{n_5}^s\LR{n_6}^s}\\
&\times \om _2(n_2)\om _3(n_3)\om _4(n_4)\om _5(n_5)\om _6(n_6)\Big| \Big\| _{\ell ^2_s}~~\lec~ \prod _{j=2}^6\tnorm{\om _j}{\ell ^2},}
where in the proof of this lemma we denote $\max _{2\le j\le 6}\LR{n_j}$ by $n_{\max}$.

Note that $\LR{n}<\LR{n_1}\sim \LR{n_{46}}\lec \LR{n_4}\sim n_{\max}$, $\LR{n}\LR{n_{23}}\lec |\Phi |$, $\LR{n_1}\LR{n_{56}}\lec |\Phi _1|\lec |\Phi +\Phi _1|$ in the above summation.
Therefore, it holds that
\eqq{&\Big| \frac{\LR{n}^s\chf{\Sc{A}_1}\tilde{m}_1(n,n_1,n_2,n_3)\tilde{m}_1(n_1,n_4,n_5,n_6)n_{max}^{-s+3/2}}{\Phi (\Phi +\Phi _1)\LR{n_2}^s\LR{n_3}^s\LR{n_4}^s\LR{n_5}^s\LR{n_6}^s}\Big| \\
&\lec \frac{\LR{n_4}^{-2s+1/2}}{\LR{n_2}^{s+1}\LR{n_3}^s\LR{n_5}^{s+1}\LR{n_6}^s\LR{n_{46}}^{1-s}}.}
This bound is sufficient in the case $1/4\le s<1/2$, since Young's inequality and twice applications of the Sobolev multiplication law yield
\eqq{&\norm{\frac{\om _2}{\LR{\cdot }^{s+1}}*\frac{\om _5}{\LR{\cdot}^{s+1}}*\frac{\om _3}{\LR{\cdot}^s}*\frac{\om _4*(\om _6/\LR{\cdot}^s)}{\LR{\cdot}^{1-s}}}{\ell ^2}\\
&\le \tnorm{\om _2}{\ell ^1_{-s-1}}\tnorm{\om_5}{\ell ^1_{-s-1}}\norm{\frac{\om _3}{\LR{\cdot}^s}*\frac{\om _4*(\om _6/\LR{\cdot}^s)}{\LR{\cdot}^{1-s}}}{\ell ^2}\\
&\lec \tnorm{\om _2}{\ell ^2}\tnorm{\om _5}{\ell ^2}\tnorm{\om _3}{\ell ^2}\norm{\om _4*\frac{\om_6}{\LR{\cdot}^{s}}}{\ell ^2_{-1/2}}\lec \prod _{j=2}^6\tnorm{\om _j}{\ell ^2}.
}
So we focus on $1/6<s<1/4$ and evaluate as
\eqq{
&\frac{\LR{n_4}^{-2s+1/2}}{\LR{n_2}^{s+1}\LR{n_3}^s\LR{n_5}^{s+1}\LR{n_6}^s\LR{n_{46}}^{1-s}}\lec \frac{\LR{n_6}^{-2s+1/2}+\LR{n_{46}}^{-2s+1/2}}{\LR{n_2}^{s+1}\LR{n_3}^s\LR{n_5}^{s+1}\LR{n_6}^s\LR{n_{46}}^{1-s}}\\
&\lec \frac{1}{\LR{n_2}^{s+1}\LR{n_3}^s\LR{n_5}^{s+1}\LR{n_6}^{3s-1/2}\LR{n_{46}}^{1-s}}+\frac{1}{\LR{n_2}^{s+1}\LR{n_3}^s\LR{n_5}^{s+1}\LR{n_6}^s\LR{n_{46}}^{s+1/2}}.}
The case corresponding to the first term gives the bound:
\eqq{&\tnorm{\om _2}{\ell ^1_{-s-1}}\tnorm{\om _5}{\ell ^1_{-s-1}}\norm{\frac{\om _3}{\LR{\cdot}^s}*\frac{\om _4*(\om _6/\LR{\cdot}^{3s-1/2})}{\LR{\cdot}^{1-s}}}{\ell ^2}\\[-5pt]
&\lec \tnorm{\om _2}{\ell ^2}\tnorm{\om _5}{\ell ^2}\tnorm{\om _3}{\ell ^2}\norm{\om _4*\frac{\om _6}{\LR{\cdot}^{3s-1/2}}}{\ell ^2_{-1/2}}\lec \prod _{j=2}^6\tnorm{\om _j}{\ell ^2},}
while for the second term we argue similarly and obtain the bound:
\eqq{&\tnorm{\om _2}{\ell ^1_{-s-1}}\tnorm{\om _5}{\ell ^1_{-s-1}}\norm{\frac{\om _3}{\LR{\cdot}^s}*\frac{\om _4*(\om _6/\LR{\cdot}^{s})}{\LR{\cdot}^{s+1/2}}}{\ell ^2}\\[-5pt]
&\lec \tnorm{\om _2}{\ell ^2}\tnorm{\om _5}{\ell ^2}\tnorm{\om _3}{\ell ^2}\norm{\om _4*\frac{\om_6}{\LR{\cdot}^{s}}}{\ell ^2_{-2s}}\lec \prod _{j=2}^6\tnorm{\om _j}{\ell ^2},}
where we have used the Sobolev multiplication law twice.
\end{proof}

As a corollary, we also obtain the estimates for $\Sc{N}_{1,0}$.

\begin{lem}\label{lem:N10}
Let $1/6<s<1/2$.
Then, there exists $\de >0$ such that we have
\eqq{\norm{\Sc{N}_{1,0}[\om ]}{C_T\ell ^2_s}&\lec M^{-\de}\tnorm{\om}{C_T\ell ^2_s}^5,\\
\norm{\Sc{N}_{1,0}[\om ]-\Sc{N}_{1,0}[\ti{\om}]}{C_T\ell ^2_s}&\lec M^{-\de}\Big( \tnorm{\om}{C_T\ell ^2_s}^4+\tnorm{\ti{\om}}{C_T\ell ^2_s}^4\Big) \tnorm{\om -\ti{\om}}{C_T\ell ^2_s}}
for any $\om ,\ti{\om}\in C_T\ell ^2_s$.
\end{lem}
\begin{proof}
These estimates follow immediately from the proof of Lemma~\ref{lem:N11}, since $n_{\max}\gec |\Phi |^{1/2}>M^{1/2}$.
\end{proof}


\subsection{Estimate for $\Sc{N}_2$}

This is an easy case, since we can control this term without further applying normal form reduction.
Recall that
\eqq{
&\Sc{N}_2[\om ](n)\\
&=i\sum _{\mat{n=n_{123}\\ |\Phi |>M}}\frac{e^{it\Phi}\tilde{m}_1(n,n_1,n_2,n_3)}{\Phi}\om (n_1)\bigg[ \sum _{n_2=n_{456}}e^{it\Phi _2}\tilde{m}_1(n_2,n_4,n_5,n_6)\om (n_4)\om (n_5)\om ^*(n_6)\bigg] \om ^*(n_3)\\[-5pt]
&~~ +i\sum _{\mat{n=n_{123}\\ |\Phi |>M}}\frac{e^{it\Phi}\tilde{m}_1(n,n_1,n_2,n_3)}{\Phi}\om (n_1)\bigg[ \sum _{n_2=n_{456}}e^{it\Phi _2}m_2(n_2,n_4,n_5,n_6)\om (n_4)\om ^*(n_5)\om (n_6)\bigg] \om ^*(n_3)\\[-5pt]
&~~ +i\sum _{\mat{n=n_{123}\\ |\Phi |>M}}\frac{e^{it\Phi}\tilde{m}_1(n,n_1,n_2,n_3)}{\Phi}\om (n_1)\bigg[ \sum _{n_2=n_{456}}e^{it\Phi _2}m_3(n_2,n_4,n_5,n_6)\om ^*(n_4)\om (n_5)\om (n_6)\bigg] \om ^*(n_3),}
where $\Phi _2:=n_2|n_2|-n_4|n_4|-n_5|n_5|-n_6|n_6|$.

\begin{lem}\label{lem:N2}
Let $1/6<s<1/2$.
Then, we have
\eqq{\norm{\Sc{N}_{2}[\om ]}{C_T\ell ^2_s}&\lec \tnorm{\om}{C_T\ell ^2_s}^5,\\
\norm{\Sc{N}_{2}[\om ]-\Sc{N}_{2}[\ti{\om}]}{C_T\ell ^2_s}&\lec \Big( \tnorm{\om}{C_T\ell ^2_s}^4+\tnorm{\ti{\om}}{C_T\ell ^2_s}^4\Big) \tnorm{\om -\ti{\om}}{C_T\ell ^2_s}}
for any $\om ,\ti{\om}\in C_T\ell ^2_s$.
\end{lem}
\begin{proof}
We focus on the estimate of the first term in $\Sc{N}_2$; noticing \eqref{bd:multiplier}, we can handle the second and the third terms similarly.
As in the proof of Lemma~\ref{lem:N1R}, we may assume $\LR{n_2}\ll \LR{n}$, and hence $\LR{n}\LR{n_{23}}\lec |\Phi |$.
Since $\LR{n},\LR{n_2},\LR{n_3}\lec \LR{n_1}$, it holds that
\eq{est2}{&\Big| \frac{\LR{n}^s\tilde{m}_1(n,n_1,n_2,n_3)\tilde{m}_1(n_2,n_4,n_5,n_6)}{\Phi}\Big| \\
&\lec \frac{\LR{n}^s}{\LR{n_1}\LR{n_5}} \lec \frac{\LR{n_1}^s\LR{n_3}^{(1/6)+}}{\LR{n_1}^{(1/2)+}\LR{n_3}^{(1/2)+}}\frac{1}{\LR{n_2}^{1/6}\LR{n_5}}.}
The claim then follows from Young's inequality and the Sobolev multiplication law:
\eqq{\norm{(\LR{\cdot}^{-1}\om _5)*(\om _4*\om _6)}{\ell ^2_{-1/6}}&\lec \tnorm{\LR{\cdot}^
{-1}\om _5}{\ell ^2_{(1/2)+}}\tnorm{\om _4*\om _6}{\ell ^2_{-1/6}}\\
&\lec \tnorm{\om _5}{\ell ^2}\tnorm{\om _4}{\ell ^2_s}\tnorm{\om _6}{\ell ^2_s},}
for $s>1/6$.
\end{proof}


\subsection{Estimate for $\Sc{N}_3$}

We recall 
\eqq{
&\Sc{N}_3[\om ](n)\\
&=i\sum _{\mat{n=n_{123}\\ |\Phi |>M}}\frac{e^{it\Phi}\tilde{m}_1(n,n_1,n_2,n_3)}{\Phi}\om (n_1)\om (n_2)\bigg[ \sum _{n_3=n_{456}}e^{it\Phi _3}\tilde{m}_1^*(n_3,n_4,n_5,n_6)\om ^*(n_4)\om ^*(n_5)\om (n_6)\bigg] \\[-5pt]
&~~ +i\sum _{\mat{n=n_{123}\\ |\Phi |>M}}\frac{e^{it\Phi}\tilde{m}_1(n,n_1,n_2,n_3)}{\Phi}\om (n_1)\om (n_2)\bigg[ \sum _{n_3=n_{456}}e^{it\Phi _3}m_2^*(n_3,n_4,n_5,n_6)\om ^*(n_4)\om (n_5)\om ^*(n_6)\bigg] \\[-5pt]
&~~ +i\sum _{\mat{n=n_{123}\\ |\Phi |>M}}\frac{e^{it\Phi}\tilde{m}_1(n,n_1,n_2,n_3)}{\Phi}\om (n_1)\om (n_2)\bigg[ \sum _{n_3=n_{456}}e^{it\Phi _3}m_3^*(n_3,n_4,n_5,n_6)\om (n_4)\om ^*(n_5)\om ^*(n_6)\bigg] ,}
where 
\eqs{m_j^*(n_3,n_4,n_5,n_6):=\bbar{m_j(-n_3,-n_4,-n_5,-n_6)},\\
\Phi _3:=n_3|n_3|-n_4|n_4|-n_5|n_5|-n_6|n_6|.}

Similarly to $\Sc{N}_{1}$, we will decompose $\Sc{N}_3$ into two parts $\Sc{N}_{3,R}+\Sc{N}_{3,N\!R}$ and apply normal form reduction to $\Sc{N}_{3,N\!R}$.
Define the set $\Sc{A}_3\subset \Shugo{(n,n_1,\dots ,n_6)\in \Bo{Z}^7}{n=n_{123},\,n_3=n_{456}}$ as
\eqq{\Sc{A}_3:=\Shugo{(n,n_j)}{\LR{n_2}\ll \LR{n}\wedge \LR{n_3},\,\LR{n_5}\ll \LR{n_3}\wedge \LR{n_6},\,|n_{25}|\ll |n_{14}|,\, |nn_{25}|\ll |n_3n_{14}|},}
and define
\eqq{&\Sc{N}_{3,R}[\om ](n)\\
&:=i\!\!\!\sum _{\mat{n=n_{123}\\ |\Phi |>M}}\chf{\LR{n_2}\gec \LR{n_3}}\sum _{n_3=n_{456}}\chf{\Sc{A}_3^c}\frac{e^{it(\Phi +\Phi _3)}\tilde{m}_1(n,n_1,n_2,n_3)}{\Phi}\om (n_1)\om (n_2)\\
&\hx\hx \times \Big[ \tilde{m}_1^*(n_3,n_4,n_5,n_6)\om ^*(n_4)\om ^*(n_5)\om (n_6)+m_2^*(n_3,n_4,n_5,n_6)\om ^*(n_4)\om (n_5)\om ^*(n_6)\\
&\hxx\hx +m_3^*(n_3,n_4,n_5,n_6)\om (n_4)\om ^*(n_5)\om ^*(n_6)\Big] \\
&\hx +i\!\!\!\sum _{\mat{n=n_{123}\\ |\Phi |>M}}\chf{\LR{n_2}\ll \LR{n_3}}\sum _{n_3=n_{456}}\chf{\Sc{A}_3^c}\frac{e^{it(\Phi +\Phi _3)}\tilde{m}_1(n,n_1,n_2,n_3)\tilde{m}_1^*(n_3,n_4,n_5,n_6)}{\Phi}\\[-10pt]
&\hspace{210pt} \times \om (n_1)\om (n_2)\om ^*(n_4)\om ^*(n_5)\om (n_6),}
\eqq{\Sc{N}_{3,N\!R}[\om ](n)&:=i\!\!\!\sum _{\mat{n=n_{123}\\ |\Phi |>M}}\sum _{n_3=n_{456}}\!\!\!\chf{\Sc{A}_3}\frac{e^{it(\Phi +\Phi _3)}\tilde{m}_1(n,n_1,n_2,n_3)\tilde{m}_1^*(n_3,n_4,n_5,n_6)}{\Phi}\\[-10pt]
&\hspace{160pt} \times \om (n_1)\om (n_2)\om ^*(n_4)\om ^*(n_5)\om (n_6).}
Note that $\tilde{m}_1(n,n_1,n_2,n_3)\neq 0$ and $\LR{n_2}\ll \LR{n_3}$ imply $n_3<0$ and $m_2^*(n_3,n_4,n_5,n_6)=m_3^*(n_3,n_4,n_5,n_6)=0$.

We first treat $\Sc{N}_{3,R}$.
\begin{lem}\label{lem:N3R}
Let $1/6<s<1/2$.
Then, we have
\eqq{\norm{\Sc{N}_{3,R}[\om ]}{C_T\ell ^2_s}&\lec \tnorm{\om}{C_T\ell ^2_s}^5,\\
\norm{\Sc{N}_{3,R}[\om ]-\Sc{N}_{3,R}[\ti{\om}]}{C_T\ell ^2_s}&\lec \Big( \tnorm{\om}{C_T\ell ^2_s}^4+\tnorm{\ti{\om}}{C_T\ell ^2_s}^4\Big) \tnorm{\om -\ti{\om}}{C_T\ell ^2_s}}
for any $\om ,\ti{\om}\in C_T\ell ^2_s$.
\end{lem}

\begin{proof}
As before, the proof has been done for the frequency region $\LR{n}\lec \LR{n_2}$.

In the case $\LR{n_3}\lec \LR{n_2}\ll \LR{n}$, the proof is exactly the same as Lemma~\ref{lem:N2} since we have a bound similar to \eqref{est2}; using \eqref{bd:multiplier},
\eqq{&\Big| \frac{\LR{n}^s\tilde{m}_1(n,n_1,n_2,n_3)m_j^*(n_3,n_4,n_5,n_6)}{\Phi}\Big| \\
&\lec \frac{\LR{n}^s}{\LR{n_1}\LR{n_{\min}}}\lec \frac{\LR{n_1}^s\LR{n_2}^{(1/6)+}}{\LR{n_1}^{(1/2)+}\LR{n_2}^{(1/2)+}}\frac{1}{\LR{n_3}^{1/6}\LR{n_{\min}}}}
for $j=1,2,3$, where $n_{\min} :=\min _{k=4,5,6}|n_k|$.

By the definition of $\Sc{A}_3$, it suffices to prove the estimate in the following three cases.

i) $\LR{n_2}\ll \LR{n}\wedge \LR{n_3}$ and $\LR{n_5}\gec \LR{n_3}\wedge \LR{n_6}$.
This case is handled similarly to Case i) of Lemma~\ref{lem:N1R}.

ii) $\LR{n_2}\ll \LR{n}\wedge \LR{n_3}$, $\LR{n_5}\ll \LR{n_3}\wedge \LR{n_6}$, and $|n_{25}|\gec |n_{14}|$.
In this case we have $\LR{n},\LR{n_3}\lec \LR{n_1}$ and $\LR{n_2},\LR{n_5},\LR{n_6}\lec \LR{n_4}$, as well as $\LR{n_{14}}\lec \LR{n_2}\vee \LR{n_5}$.
Therefore,
\eqq{&\Big| \frac{\LR{n}^s\tilde{m}_1(n,n_1,n_2,n_3)\tilde{m}_1(n_3,n_4,n_5,n_6)}{\Phi}\Big| \lec \frac{\LR{n}^s}{\LR{n_2}\LR{n_5}}\lec \frac{\LR{n_1}^s}{\LR{n_{14}}[\LR{n_2}\wedge \LR{n_5}]}\\
&\lec \frac{\LR{n_1}^s}{\LR{n_{14}}[\LR{n_2}\wedge \LR{n_5}]}\frac{\LR{n_4}^{(1/6)+}[\LR{n_2}\vee \LR{n_5}]^{(1/6)+}\LR{n_6}^{(1/6)+}}{\LR{n_4}^{0+}[(\LR{n_2}\vee \LR{n_5})\wedge \LR{n_6}]^{(1/2)+}},}
which yields the claim if $s>1/6$ by use of Young's inequality.

iii) $\LR{n_2}\ll \LR{n}\wedge \LR{n_3}$, $\LR{n_5}\ll \LR{n_3}\wedge \LR{n_6}$, and $|nn_{25}|\gec |n_3n_{14}|$.
It holds that
\eqq{&\Big| \frac{\LR{n}^s\tilde{m}_1(n,n_1,n_2,n_3)\tilde{m}_1(n_3,n_4,n_5,n_6)}{\Phi}\Big| \\
&\lec \frac{\LR{n}^s\LR{n_3}}{\LR{n_1}\LR{n_2}\LR{n_5}}\lec \frac{\LR{n_1}^s\LR{n_{25}}}{\LR{n_{14}}\LR{n_2}\LR{n_5}}\lec \frac{\LR{n_1}^s}{\LR{n_{14}}[\LR{n_2}\wedge \LR{n_5}]},}
and we argue in the same manner as ii).
\end{proof}

For $\Sc{N}_{3,N\!R}$, it turns out that $|\Phi +\Phi _3|$ has a lower bound.
In fact, frequencies in the sum satisfy
\eqq{\Phi +\Phi _3&=n^2-n_1^2-n_2|n_2|+n_4^2-n_5|n_5|-n_6^2\\
&=2n_{14}n_{46}+2nn_{25}-n_{25}^2-n_2|n_2|-n_5|n_5|,}
while it holds that
\eqs{|nn_{25}|\ll |n_3n_{14}|\sim |n_{14}n_{46}|,\\
\big| n_{25}^2+n_2|n_2|+n_5|n_5|\big| \lec |n_{25}|\big( |n_2|\vee |n_5|\big) \ll |n_3n_{25}|\ll |n_3||n_{14}|,}
which imply $|\Phi +\Phi _3|\gec |n_3||n_{14}|$.
Note that we may assume $|n_3n_{14}|\neq 0$; otherwise, the argument for the case $|n_{25}|\gec |n_{14}|$ may be applied since we always have $n_3<0$.

We apply a differentiation by parts to $\Sc{N}_{3,N\!R}$ as
\eqq{&\Sc{N}_{3,N\!R}[\om ](n)\\
&=\p _t\bigg[ \sum _{\mat{n=n_{123}\\ |\Phi |>M}}\sum _{n_3=n_{456}}\frac{e^{it(\Phi +\Phi _3)}\chf{\Sc{A}_3}\tilde{m}_1(n,n_1,n_2,n_3)\tilde{m}_1^*(n_3,n_4,n_5,n_6)}{\Phi (\Phi +\Phi _3)}\\[-10pt]
&\hspace{190pt} \times \om (n_1)\om (n_2)\om ^*(n_4)\om ^*(n_5)\om (n_6)\bigg] \\
&\hx -\sum _{\mat{n=n_{123}\\ |\Phi |>M}}\sum _{n_3=n_{456}}\frac{e^{it(\Phi +\Phi _3)}\chf{\Sc{A}_3}\tilde{m}_1(n,n_1,n_2,n_3)\tilde{m}_1^*(n_3,n_4,n_5,n_6)}{\Phi (\Phi +\Phi _3)}\\[-10pt]
&\hspace{180pt} \times \p _t\Big[ \om (n_1)\om (n_2)\om ^*(n_4)\om ^*(n_5)\om (n_6)\Big] \\
&=:\p _t\Sc{N}_{3,0}[\om ](n)+\Sc{N}_{3,1}[\om ](n),}
where, as $\Sc{N}_{1,N\!R}$, the formal computations are justified by the absolute convergence of the sum in $\Sc{N}_{3,0}$.

\begin{lem}\label{lem:N31}
Let $1/6<s<1/2$.
Then, we have
\eqq{&\norm{\Sc{N}_{3,1}[\om ]}{C_T\ell ^2_s}\lec \Big( 1+\tnorm{u}{C_TH^s}^4+\tnorm{\om}{C_T\ell ^2_s}^3\Big) \tnorm{\om}{C_T\ell ^2_s}^4,\\
&\norm{\Sc{N}_{3,1}[\om ]-\Sc{N}_{3,1}[\ti{\om}]}{C_T\ell ^2_s}\\
&\hx \lec \Big( 1+\tnorm{u}{C_TH^s}^4+\tnorm{\ti{u}}{C_TH^s}^4\Big) \Big( \tnorm{\om}{C_T\ell ^2_s}^4+\tnorm{\ti{\om}}{C_T\ell ^2_s}^4\Big) \tnorm{u-\ti{u}}{C_TH^s}\\
&\hx \hx +\Big( 1+\tnorm{u}{C_TH^s}^4+\tnorm{\ti{u}}{C_TH^s}^4+\tnorm{\om}{C_T\ell ^2_s}^3+\tnorm{\ti{\om}}{C_T\ell ^2_s}^3\Big) \Big( \tnorm{\om}{C_T\ell ^2_s}^3+\tnorm{\ti{\om}}{C_T\ell ^2_s}^3\Big) \tnorm{\om -\ti{\om}}{C_T\ell ^2_s}}
for any real-valued mean-zero solutions $u,\ti{u}\in C_TH^s$ of \eqref{BO} and the corresponding solutions $\om ,\ti{\om}\in C_T\ell ^2_s$ of \eqref{eq_om} defined by \eqref{utov}, \eqref{vtoom}.
\end{lem}

\begin{proof}
Similarly to Lemma~\ref{lem:N11}, it suffices to show that
\eqq{&\norm{\sum _{\mat{n=n_{123}\\ |\Phi |>M}}\sum _{n_3=n_{456}}\!\!\!\!\Big| \frac{\chf{\Sc{A}_3}\tilde{m}_1(n,n_1,n_2,n_3)\tilde{m}_1^*(n_3,n_4,n_5,n_6)n_{\max}^{\frac{3}{2}-s}}{\Phi (\Phi +\Phi _3)\LR{n_1}^s\LR{n_2}^s\LR{n_4}^s\LR{n_5}^s\LR{n_6}^s}\om _1(n_1)\om _2(n_2)\om _4(n_4)\om _5(n_5)\om _6(n_6)\Big|}{\ell ^2_s}\\
&\lec \prod _{j=1,2,4,5,6}\tnorm{\om _j}{\ell ^2},}
where in the proof of this lemma we denote $\max _{j=1,2,4,5,6}\LR{n_j}$ by $n_{\max}$.

Now, we have $n_{\max}\sim \LR{n_1}\vee \LR{n_4}$ and, since $1/6<s<1/2$,
\eqq{&\Big| \frac{\LR{n}^s\chf{\Sc{A}_1}\tilde{m}_1(n,n_1,n_2,n_3)\tilde{m}_1^*(n_3,n_4,n_5,n_6)n_{max}^{\frac{3}{2}-s}}{\Phi (\Phi +\Phi _3)\LR{n_1}^s\LR{n_2}^s\LR{n_4}^s\LR{n_5}^s\LR{n_6}^s}\Big| \lec \frac{n_{\max}^{\frac{3}{2}-s}}{\LR{n_1}\LR{n_2}^{s+1}\LR{n_4}^{2s}\LR{n_5}^{s+1}\LR{n_{14}}}\\
&\sim \begin{cases}
\dfrac{1}{\LR{n_1}^{s+\frac{1}{2}}\LR{n_2}^{s+1}\LR{n_4}^{2s}\LR{n_5}^{s+1}}\lec \dfrac{1}{\LR{n_1}^{\frac{1}{2}+}\LR{n_2}^{s+1}\LR{n_4}^{\frac{1}{2}+}\LR{n_5}^{s+1}} &\text{if $\LR{n_1}\gg \LR{n_4}$},\\[15pt]
\dfrac{1}{\LR{n_1}^{3s-\frac{1}{2}}\LR{n_2}^{s+1}\LR{n_5}^{s+1}\LR{n_{14}}}\lec \dfrac{1}{\LR{n_2}^{s+1}\LR{n_5}^{s+1}\LR{n_{14}}^{1+}} &\text{if $\LR{n_1}\sim \LR{n_4}$},\\[15pt]
\dfrac{1}{\LR{n_1}\LR{n_2}^{s+1}\LR{n_4}^{3s-\frac{1}{2}}\LR{n_5}^{s+1}}\lec \dfrac{1}{\LR{n_1}^{\frac{1}{2}+}\LR{n_2}^{s+1}\LR{n_5}^{s+1}\LR{n_{46}}^{\frac{1}{2}+}} &\text{if $\LR{n_1}\ll \LR{n_4}$},
\end{cases}}
where in the case $\LR{n_1}\ll \LR{n_4}$ we have used the fact that $\LR{n_{46}}\sim \LR{n_3}\lec \LR{n_1}$.
Then, the claimed estimates are deduced by Young's inequality.
\end{proof}

Finally, the estimates for $\Sc{N}_{3,0}$ follow from Lemma~\ref{lem:N31} just as Lemma~\ref{lem:N10} follows from Lemma~\ref{lem:N11}.
\begin{lem}\label{lem:N30}
Let $1/6<s<1/2$.
Then, there exists $\de >0$ such that we have
\eqq{\norm{\Sc{N}_{3,0}[\om ]}{C_T\ell ^2_s}&\lec M^{-\de}\tnorm{\om}{C_T\ell ^2_s}^5,\\
\norm{\Sc{N}_{3,0}[\om ]-\Sc{N}_{3,0}[\ti{\om}]}{C_T\ell ^2_s}&\lec M^{-\de}\Big( \tnorm{\om}{C_T\ell ^2_s}^4+\tnorm{\ti{\om}}{C_T\ell ^2_s}^4\Big) \tnorm{\om -\ti{\om}}{C_T\ell ^2_s}}
for any $\om ,\ti{\om}\in C_T\ell ^2_s$.
\end{lem}


\subsection{Conclusion}

We have seen that a real-valued mean-zero solution $u\in C_TH^s$ of \eqref{BO} with $1/6<s<1/2$ and the corresponding solution $\om \in C_T\ell ^2_s$ of \eqref{eq_om} defined by \eqref{utov} and \eqref{vtoom} satisfy the equation
\eqq{\p _t\om (t,n)=\p _t\Sc{N}^{(0)}[\om ](t,n)+\Sc{N}^{(1)}[u,\om ](t,n),\qquad t\in [0,T]}
with
\eqq{\Sc{N}^{(0)}[\om ]&:=\Sc{N}_0[\om ]+\Sc{N}_{1,0}[\om ]+\Sc{N}_{3,0}[\om ],\\
\Sc{N}^{(1)}[u,\om ]&:=\Sc{R} [u,\om ]+\Sc{N}_R[\om ]+\Sc{R}_1[u,\om ]+\Sc{N}_{1,R}[\om ]\\
&\qquad +\Sc{N}_{1,1}[\om ]+\Sc{N}_2[\om ]+\Sc{N}_{3,R}[\om ]+\Sc{N}_{3,1}[\om ],}
or the associated integral equation
\eq{eq_omZ}{\om (t,n)&=\om (0,n)+\Big[ \Sc{N}^{(0)}[\om ](t,n)-\Sc{N}^{(0)}[\om ](0,n)\Big] \\
&\hx +\int _0^t\Sc{N}^{(1)}[u,\om ](\tau ,n)\,d\tau ,\qquad t\in [0,T],}
both in the classical sense for each $n>0$.

Combining Lemmas~\ref{lem:R}--\ref{lem:N0} and \ref{lem:N1R}--\ref{lem:N30}, we obtain the following estimates.
\begin{prop}\label{prop}
Let $1/6<s<1/2$ and $u,\ti{u}\in C_TH^s$ be any real-valued mean-zero solutions of \eqref{BO}, with $\om ,\ti{\om}\in C_T\ell ^2_s$ being the corresponding solutions of \eqref{eq_om} defined by \eqref{utov} and \eqref{vtoom}.
Then, there exist $\de =\de (s)>0$ and $C=C(s,\tnorm{u}{C_TH^s},\tnorm{\ti{u}}{C_TH^s})>0$%
\footnote{From Lemma~\ref{lem:exp}, $\tnorm{\om}{C_T\ell ^2_s}$ is controlled by $\tnorm{u}{C_TH^s}$.} 
such that we have
\eqq{\norm{\Sc{N}^{(0)}[\om ]-\Sc{N}^{(0)}[\ti{\om}]}{C_T\ell ^2_s}&\le CM^{-\de}\tnorm{\om -\ti{\om}}{C_T\ell ^2_s},\\
\norm{\Sc{N}^{(1)}[u,\om ]-\Sc{N}^{(1)}[\ti{u},\ti{\om}]}{C_T\ell ^2_s}&\le CM\Big( \tnorm{u-\ti{u}}{C_TH^s}+\tnorm{\om -\ti{\om}}{C_T\ell ^2_s}\Big) .}
\end{prop}


\bigskip
\section{Proof of Theorem~\ref{thm}}\label{sec:proof}

It suffices to verify the claim for $1/6<s<1/2$.
Let $u,\ti{u}\in C_TH^s$ be two real-valued mean-zero solutions of \eqref{BO} on a time interval $[0,T]$ with a common initial datum at $t=0$.
Define the corresponding functions $V,\ti{V}\in C_TH^{s+1}$, $v,\ti{v}\in C_TH^s$, and $\om ,\ti{\om}\in C_T\ell ^2_s$ by \eqref{utov} and \eqref{vtoom}.
By a continuity argument, it suffices to prove $u(t)=\ti{u}(t)$ on $[0,T']$ for some $0<T'\le T$.
In the following, $\ti{C}>0$ denotes any constant depending on $s,\tnorm{u}{C_TH^s},\tnorm{\ti{u}}{C_TH^s}$.

Let $N>0$ be a large constant to be chosen later.
We estimate $P_{\le N}(u-\ti{u})$ and $P_{>N}(u-\ti{u})$ separately.
For the low frequency part, which is smooth and satisfies the integral equation
\eqq{P_{\le N}(u-\ti{u})(t)=\int _0^te^{-(t-\tau )\H \p _x^2}P_{\le N}\p _x(u^2-\ti{u}^2)(\tau )\,d\tau ,}
we use the Sobolev multiplication law and obtain
\eqq{\norm{P_{\le N}(u-\ti{u})}{C_{T'}H^s}&\le \int _0^{T'}\norm{P_{\le N}\p _x(u^2-\ti{u}^2)(t)}{H^s}\,dt\\
&\le \ti{C}T'N^2\tnorm{u-\ti{u}}{C_{T'}H^s}}
for any $0<T'\le T$.

For the high frequency part, we use \eqref{u_via_v} to have
\eqq{&\norm{P_{>N}(u-\ti{u})}{C_{T'}H^s}=2\norm{P_{>N}P_+(u-\ti{u})}{C_{T'}H^s}\\
&\le 2\norm{P_{>N}P_+(\bbar{V}P_+v-\bbar{\ti{V}}P_+\ti{v})}{C_{T'}H^s}+2\norm{P_{>N}P_+(\bbar{V}P_-v-\bbar{\ti{V}}P_-\ti{v})}{C_{T'}H^s}.
}
By Lemma~\ref{lem:exp}, we estimate the second term as
\eqq{&\norm{P_{>N}P_+(\bbar{V}P_-v-\bbar{\ti{V}}P_-\ti{v})}{C_{T'}H^s}\le \norm{P_{>N}P_+\bbar{V}\cdot P_-v-P_{>N}P_+\bbar{\ti{V}}\cdot P_-\ti{v}}{C_{T'}H^s}\\
&\le \ti{C}\Big( \norm{P_{>N}(V-\ti{V})}{C_{T'}H^{1}}\tnorm{v}{C_{T'}H^s}+\norm{P_{>N}\ti{V}}{C_{T'}H^{1}}\tnorm{v-\ti{v}}{C_{T'}H^s}\Big) \\
&\le \ti{C}N^{-s}\Big( \tnorm{V-\ti{V}}{C_{T'}H^{s+1}}\tnorm{v}{C_{T'}H^s}+\tnorm{\ti{V}}{C_{T'}H^{s+1}}\tnorm{v-\ti{v}}{C_{T'}H^s}\Big) \\
&\le \ti{C}N^{-s}\tnorm{u-\ti{u}}{C_{T'}H^s}.}
The first term is estimated as
\eqq{&\norm{P_{>N}P_+(\bbar{V}P_+v-\bbar{\ti{V}}P_+\ti{v})}{C_{T'}H^s}\\
&\le \norm{P_{>N/2}(V-\ti{V})}{C_{T'}H^{1}}\tnorm{v}{C_{T'}H^s}+\tnorm{V-\ti{V}}{C_{T'}H^{1}}\norm{P_{>N/2}v}{C_{T'}H^s}\\
&\hx +\tnorm{\ti{V}}{C_{T'}H^{1}}\norm{P_+(v-\ti{v})}{C_{T'}H^s}\\
&\le \ti{C}\Big( \big( N^{-s}+\norm{P_{>N/2}v}{C_{T'}H^s}\big) \tnorm{u-\ti{u}}{C_{T'}H^s}+\tnorm{\chf{n>0}(\om -\ti{\om})}{C_{T'}\ell ^2_s}\Big) .}
Using the equation \eqref{eq_omZ} and Proposition~\ref{prop}, we have
\eqq{\tnorm{\chf{n>0}(\om -\ti{\om})}{C_{T'}\ell ^2_s}&\le \ti{C}\big( M^{-\de}+T'M\big) \Big( \tnorm{u-\ti{u}}{C_{T'}H^s}+\tnorm{\om -\ti{\om}}{C_{T'}\ell ^2_s}\Big) \\
&\le \ti{C}\big( M^{-\de}+T'M\big) \tnorm{u-\ti{u}}{C_{T'}H^s}.}
Hence, we obtain the estimate
\eqq{&\tnorm{u-\ti{u}}{C_{T'}H^s}\\
&\le \ti{C}\Big( T'\big( N^2+M\big) +N^{-s}+\norm{P_{>N/2}v}{C_{T'}H^s}+M^{-\de}\Big) \tnorm{u-\ti{u}}{C_{T'}H^s}}
for any $0<T'\le T$.
We take $N,M>0$ sufficiently large and then take $T'$ small according to $N,M$ so that
\eqq{\ti{C}\Big( T'\big( N^2+M\big) +N^{-s}+\norm{P_{>N/2}v}{C_{T}H^s}+M^{-\de}\Big) <1.}
For such a $T'$ we obtain $\tnorm{u-\ti{u}}{C_{T'}H^s}=0$ as desired.
This concludes the proof of Theorem~\ref{thm}.


\section*{Acknowledgements}
The result in this article was obtained while the author was visiting the University of Chicago in 2013--2014.
He would like to express his deep gratitude to Professor Carlos E. Kenig for fruitful discussions and also to all members of the Department of Mathematics for their heartwarming hospitality.
This work is partially supported by Japan Society for the Promotion of Science (JSPS) [grant numbers JP24740086, JP16K17626].



\begin{thebibliography}{10}

\bibitem{BMN97}
A.~Babin, A.~Mahalov, and B.~Nicolaenko.
\newblock Regularity and integrability of {$3$}{D} {E}uler and
  {N}avier-{S}tokes equations for rotating fluids.
\newblock {\em Asymptot. Anal.}, 15(2):103--150, 1997.

\bibitem{BIT11}
A.~V. Babin, A.~A. Ilyin, and E.~S. Titi.
\newblock On the regularization mechanism for the periodic {K}orteweg-de
  {V}ries equation.
\newblock {\em Comm. Pure Appl. Math.}, 64(5):591--648, 2011.

\bibitem{BP08}
N.~Burq and F.~Planchon.
\newblock On well-posedness for the {B}enjamin-{O}no equation.
\newblock {\em Math. Ann.}, 340(3):497--542, 2008.

\bibitem{GKT20p}
P.~G\'{e}rard, T.~Kappeler, and P.~Topalov.
\newblock Sharp well-posedness results of the {B}enjamin-{O}no equation in
  {$H^s(\mathbb{T},\mathbb{R})$} and qualitative properties of its solution.
\newblock {\em preprint}, 2020.

\bibitem{GKO13}
Z.~Guo, S.~Kwon, and T.~Oh.
\newblock Poincar\'{e}-{D}ulac normal form reduction for unconditional
  well-posedness of the periodic cubic {NLS}.
\newblock {\em Comm. Math. Phys.}, 322(1):19--48, 2013.

\bibitem{GLM14}
Z.~Guo, Y.~Lin, and L.~Molinet.
\newblock Well-posedness in energy space for the periodic modified
  {B}enjamin-{O}no equation.
\newblock {\em J. Differential Equations}, 256(8):2778--2806, 2014.

\bibitem{IK07}
A.~D. Ionescu and C.~E. Kenig.
\newblock Global well-posedness of the {B}enjamin-{O}no equation in
  low-regularity spaces.
\newblock {\em J. Amer. Math. Soc.}, 20(3):753--798, 2007.

\bibitem{KK03}
C.~E. Kenig and K.~D. Koenig.
\newblock On the local well-posedness of the {B}enjamin-{O}no and modified
  {B}enjamin-{O}no equations.
\newblock {\em Math. Res. Lett.}, 10(5-6):879--895, 2003.

\bibitem{K-all}
N.~Kishimoto.
\newblock Unconditional uniqueness of solutions for nonlinear dispersive
  equations.
\newblock {\em preprint}, 2019.

\bibitem{KT03}
H.~Koch and N.~Tzvetkov.
\newblock On the local well-posedness of the {B}enjamin-{O}no equation in
  {$H^s(\mathbb{R})$}.
\newblock {\em Int. Math. Res. Not.}, (26):1449--1464, 2003.

\bibitem{KO12}
S.~Kwon and T.~Oh.
\newblock On unconditional well-posedness of modified {K}d{V}.
\newblock {\em Int. Math. Res. Not. IMRN}, (15):3509--3534, 2012.

\bibitem{KOY20}
S.~Kwon, T.~Oh, and H.~Yoon.
\newblock Normal form approach to unconditional well-posedness of nonlinear
  dispersive {PDE}s on the real line.
\newblock {\em Ann. Fac. Sci. Toulouse Math. (6)}, 29(3):649--720, 2020.

\bibitem{M07}
L.~Molinet.
\newblock Global well-posedness in the energy space for the {B}enjamin-{O}no
  equation on the circle.
\newblock {\em Math. Ann.}, 337(2):353--383, 2007.

\bibitem{M08}
L.~Molinet.
\newblock Global well-posedness in {$L^2$} for the periodic {B}enjamin-{O}no
  equation.
\newblock {\em Amer. J. Math.}, 130(3):635--683, 2008.

\bibitem{MP12}
L.~Molinet and D.~Pilod.
\newblock The {C}auchy problem for the {B}enjamin-{O}no equation in {$L^2$}
  revisited.
\newblock {\em Anal. PDE}, 5(2):365--395, 2012.

\bibitem{P91}
G.~Ponce.
\newblock On the global well-posedness of the {B}enjamin-{O}no equation.
\newblock {\em Differential Integral Equations}, 4(3):527--542, 1991.

\bibitem{TT04}
H.~Takaoka and Y.~Tsutsumi.
\newblock Well-posedness of the {C}auchy problem for the modified {K}d{V}
  equation with periodic boundary condition.
\newblock {\em Int. Math. Res. Not.}, (56):3009--3040, 2004.

\bibitem{T04}
T.~Tao.
\newblock Global well-posedness of the {B}enjamin-{O}no equation in {$H^1(\mathbf{R})$}.
\newblock {\em J. Hyperbolic Differ. Equ.}, 1(1):27--49, 2004.

\end{thebibliography}
\end{document}